\documentclass[11pt, reqno]{amsart}
\usepackage{amsmath, amsfonts, amsthm, amssymb, multicol, mathtools, dsfont, mathrsfs}
\usepackage{graphicx}
\usepackage{float, hyperref}
\usepackage{scalerel,stackengine, subcaption}
\usepackage[dvipsnames,dvipsnames,x11names]{xcolor}
\usepackage{enumitem}
\usepackage{pgfplots}
\usepgfplotslibrary{fillbetween}
\pgfplotsset{width=10cm,compat=1.9}
\usepackage[square,sort,comma,numbers]{natbib}
\usepackage{comment}
\usepackage{soul}
\usepackage[normalem]{ulem}

\DeclareFontFamily{U}{mathx}{\hyphenchar\font45}
\DeclareFontShape{U}{mathx}{m}{n}{<-> mathx10}{}
\DeclareSymbolFont{mathx}{U}{mathx}{m}{n}
\DeclareMathAccent{\widebar}{0}{mathx}{"73}

\allowdisplaybreaks

\makeatletter
\g@addto@macro\bfseries{\boldmath}
\makeatother

\makeatletter
\def\@setauthors{%
  \begingroup
  \def\thanks{\protect\thanks@warning}%
  \trivlist
  \centering\footnotesize \@topsep30\p@\relax
  \advance\@topsep by -\baselineskip
  \item\relax
  \author@andify\authors
  \def\\{\protect\linebreak}

  \normalsize\lowercase{\authors}%
  
	\ifx\@empty\contribs
  \else
    ,\penalty-3 \space \@setcontribs
    \@closetoccontribs
  \fi
  \endtrivlist
  \endgroup
}
\def\@settitle{\begin{center}
\LARGE\lowercase{\@title}
  \end{center}%
}
\makeatother
\newcommand{\authoremail}[1]{\email{\href{mailto:#1}{\color{lightblue}{#1}}}}
\newcommand{\authoraddress}[1]{\address{\normalfont{#1}}}

\numberwithin{equation}{section}

\newtheorem{thm}{Theorem}[section]
\newtheorem{lemma}[thm]{Lemma}
\newtheorem{cor}[thm]{Corollary}

\newtheorem{prop}[thm]{Proposition}

\theoremstyle{remark}
\newtheorem{definition}[thm]{Definition}
\newtheorem{remark}[thm]{Remark}

\renewcommand{\epsilon}{\varepsilon}

\renewcommand{\ge}{\geqslant}

\renewcommand{\geq}{\geqslant}
\renewcommand{\leq}{\leqslant}

\makeatletter
\DeclareRobustCommand\widecheck[1]{{\mathpalette\@widecheck{#1}}}
\def\@widecheck#1#2{%
    \setbox\z@\hbox{\m@th$#1#2$}%
    \setbox\tw@\hbox{\m@th$#1%
       \widehat{%
          \vrule\@width\z@\@height\ht\z@
          \vrule\@height\z@\@width\wd\z@}$}%
    \dp\tw@-\ht\z@
    \@tempdima\ht\z@ \advance\@tempdima2\ht\tw@ \divide\@tempdima\thr@@
    \setbox\tw@\hbox{%
       \raise\@tempdima\hbox{\scalebox{1}[-1]{\lower\@tempdima\box
\tw@}}}%
    {\ooalign{\box\tw@ \cr \box\z@}}}
\makeatother

\stackMath
\newcommand\reallywidehat[1]{%
\savestack{\tmpbox}{\stretchto{%
  \scaleto{%
    \scalerel*[\widthof{\ensuremath{#1}}]{\kern.1pt\mathchar"0362\kern.1pt}%
    {\rule{0ex}{\textheight}}
  }{\textheight}% 
}{2.4ex}}%
\stackon[-6.9pt]{#1}{\tmpbox}%
}
\definecolor{lightblue}{HTML}{2B77A4}
\colorlet{plotblue}{LightSkyBlue3!80}
\definecolor{darkred}{HTML}{9E0D0D}
\definecolor{purp}{HTML}{d603a9}
\definecolor{dartmouthgreen}{HTML}{00A64F}
\definecolor{Junglegreen}{HTML}{00A99A}
\definecolor{yellowcolour}{HTML}{f07c02}
\hypersetup{
	colorlinks=true,
	linkcolor=darkred,
	urlcolor=darkred,
	citecolor=lightblue
}
\title{Majorization and additive tuples in $\mathbb{Z}_2^n$}

\usepackage{fancyhdr}
\author{\large{Sophie Huczynska}}
\authoraddress{Sophie Huczynska, School of Mathematics and Statistics, University of St Andrews, St Andrews, KY16 9SS, Scotland}
\authoremail{sh70@st-andrews.ac.uk}
\thanks{}

\author{\large{Firdavs Rakhmonov}}
\authoraddress{Firdavs Rakhmonov, School of Mathematics and Statistics, University of St Andrews, St Andrews, KY16 9SS, Scotland}
\authoremail{fr52@st-andrews.ac.uk}
\thanks{FR was  financially supported by a  \emph{Leverhulme Trust Research Project Grant} (RPG-2023-281).}

\author{\large{Chi Hoi Yip}}
\authoraddress{Chi Hoi Yip, Department of Mathematics, Hong Kong University of Science and Technology, Clear Water Bay, Hong Kong}
\authoremail{machyip@ust.hk}
\thanks{}

\date{}

\begin{document}
\thispagestyle{empty}

\begin{abstract}
Majorization is a fundamental tool for comparing how ``spread out'' the entries of two vectors are. Key majorization results were obtained for the integers by Hardy, Littlewood and P\'olya and for $\mathbb{Z}_p$ by Lev. In this paper, we establish a powerful majorization theorem in $\mathbb{Z}_2^n$ that is an analogue of Lev's result in $\mathbb{Z}_p$. Our proof is based on a novel use of a compression argument that optimizes certain sums of additive representation counts. Our majorization theorem has several applications in additive combinatorics. We resolve the question: given subsets $A_1,\ldots,A_k\subseteq\mathbb{Z}_2^n$ of prescribed sizes, when is the function $r_k(A_1,\ldots,A_k)$, which counts the number of additive $k$-tuples in $A_1\times\cdots\times A_k$, maximized? We establish the corresponding minimization result and a characterization of all extremizers for $r_k(A,\ldots,A)$ when $k$ is odd. When $k=3$, this quantity is the number of Schur triples in $A$; as a special case, we recover a theorem of Samotij and Sudakov on Schur triples. We also obtain a convexity inequality and use it to prove an analogue of Pollard's theorem that strengthens and extends a well-known result of Bollob\'as and Leader in the $\mathbb{Z}_2^n$ setting. 
\end{abstract}
\maketitle
%\tableofcontents

\section{Introduction}

Majorization is a concept that arises naturally in a range of mathematical disciplines in which inequalities play a key role. It emerged at the beginning of the twentieth century through the work of several independent pioneers in different fields. Muirhead \cite{Mui03} investigated generalizations of the arithmetic-geometric mean inequality; Hardy, Littlewood, and P\'olya \cite{HLP29} studied the theory of rearrangements; and Schur \cite{Sch23} established majorization results for positive semidefinite Hermitian matrices. Economists were also among the early contributors, motivated by the need for rigorous methods to compare how equally income and wealth are distributed. For a comprehensive account of majorization, we refer the reader to the book by Marshall, Olkin, and Arnold \cite{MOA11}.

Intuitively, given a function $f:\mathbb{R}^m\to\mathbb{R}$, one seeks conditions on vectors $\mathbf{x}=(x_1,\ldots,x_m)\in\mathbb{R}^m$ and $\mathbf{y}=(y_1,\ldots,y_m)\in\mathbb{R}^m$ that capture the idea that the entries of $\mathbf{x}$ are ``less spread out'' than those of $\mathbf{y}$ and ensure that
$$
f(\mathbf{x})\leq f(\mathbf{y}).
$$
It turns out that, in many situations, the appropriate notion is the following.

\begin{definition}
\label{def:majorization}
Let $\mathbf{x}=(x_1,\dots,x_m), \mathbf{y}=(y_1,\dots,y_m)\in \mathbb{R}^m$. We say that $\mathbf{x}$ is majorized by $\mathbf{y}$, written $\mathbf{x} \preccurlyeq \mathbf{y}$, if the following two conditions hold:
    \begin{enumerate}[label=\arabic*)]
        \item The entries of both vectors are arranged in nonincreasing order: $x_1\geq \dots\geq x_m$ and $y_1\geq \dots \geq y_m$.
        \item We have $\sum\limits_{i=1}^{\ell} x_i \leq \sum\limits_{i=1}^{\ell} y_i$ for every $1\leq \ell \leq m-1$, and $\sum\limits_{i=1}^m x_i = \sum\limits_{i=1}^m y_i$.
    \end{enumerate}
\end{definition}
 
In \cite{HLP29}, Hardy, Littlewood and P\'olya sought conditions ensuring that
$\sum_{i=1}^m f(x_i) \leq \sum_{i=1}^m f(y_i)$
for every convex function $f: \mathbb{R} \rightarrow \mathbb{R}$. They proved that this inequality holds precisely when $\mathbf{x}$ is majorized by $\mathbf{y}$. Independently, Karamata \cite{Kar32} established the following variant, now known as Karamata's inequality.
\begin{prop}[Karamata's inequality]
\label{prop:karam}
Let $I\subseteq\mathbb{R}$ be an interval, and let
$(x_1,\dots,x_m)\in I^m$ and $(y_1,\dots,y_m)\in I^m$ satisfy $(x_1,\dots,x_m)\preccurlyeq(y_1,\dots,y_m)$. If $f:I\to\mathbb{R}$ is convex, then
$$
f(x_1)+\dots+f(x_m)
\leq
f(y_1)+\dots+f(y_m).
$$
\end{prop}
Classical majorization theorems based on rearrangement were obtained for the integers by Hardy, Littlewood and P\'olya (see Chapter X of \cite{HLP34}) and Gabriel \cite{Gab32}. In 2001, Lev \cite{Lev01} proved analogues in $\mathbb{Z}_p$ of a sequence of three such results, each of which is used to establish the next (Theorems 1, 2 and 3). Lev observed that, whereas the $\mathbb{Z}_p$ versions readily imply the corresponding integer results, the reverse implication does not hold. For the first majorization result in \cite{Lev01} (Theorem 1), the first part, from which the second part can be deduced, is proved by showing that it is equivalent to Pollard's theorem. Lev also noted in the abstract of \cite{Lev01} that Theorem 1 has the following consequence: if $A_1,\dots,A_k\subseteq \mathbb{Z}_p$, then the number of solutions to
\begin{equation*}
    c_1x_1+\cdots+c_kx_k=\lambda,
    \qquad (x_1,\dots,x_k)\in A_1\times\dots\times A_k,
\end{equation*}
where $c_i\in \mathbb{Z}_p\setminus\{0\}$ and $\lambda\in \mathbb{Z}_p$ are fixed, does not exceed the number of solutions to
\begin{equation*}
    x_1+\cdots+x_k=0,
    \qquad (x_1,\dots,x_k)\in \widetilde{A}_1\times\dots\times \widetilde{A}_k,
\end{equation*}
where each $\widetilde{A}_i$ is an arithmetic progression in $\mathbb{Z}_p$ of size $|A_i|$ centered at zero.  This, in turn, yields a maximization result for additive $k$-tuples in $\mathbb{Z}_p$. The same result was independently established in 2019 by Chervak, Pikhurko and Staden~\cite{CPS19}, also using Pollard's theorem; these authors were unaware of Lev's work prior to publication. More generally, although ideas related to majorization appear implicitly throughout the additive combinatorics literature, the connection does not seem to have been explicitly recognized by either community. For example, while \cite{MOA11} contains several references to applications in combinatorics, these mainly concern graph theory, and no mention is made of the results of Lev or Pollard described above.

In this paper, we prove a majorization theorem in $\mathbb{Z}_2^n$ analogous to Theorem 1 of \cite{Lev01}. Throughout the paper, we write $q\coloneqq |\mathbb{Z}_2^n|=2^n$. We identify $\mathbb{Z}_2^n$ with the binary strings of length $n$, ordered by increasing binary value. For $0\leq a\leq q$, let $\mathrm{IS}_a$ denote the initial segment of $\mathbb{Z}_2^n$ of size $a$. For $\mathbf{x} \in \mathbb{Z}_2^n$ and subsets $A_1,\ldots, A_k \subseteq \mathbb{Z}_2^n$, we write $r_{A_1+\dots+A_k}(\mathbf{x})$ for the number of ordered tuples in $A_1 \times \cdots \times A_k$ whose entries sum to $\mathbf{x}$.  These concepts are defined formally in Section~\ref{sec:prelim}.

\begin{thm}\label{thm:majorization}
Let $0\leq a_1,\dots,a_k\leq q$, and let $A_1,\dots,A_k\subseteq \mathbb{Z}_2^n$ satisfy $|A_i|=a_i$ for every $i\in [k]$. Then, for every
$B\subseteq\mathbb Z_2^n$, writing $b\coloneqq |B|$, we have
\begin{equation}
\label{eq:setwise-rearrangement}
    \sum_{\mathbf{x}\in B}
    r_{A_1+\cdots+A_k}(\mathbf{x})
    \leq
    \sum_{\mathbf{x}\in\mathrm{IS}_{b}}
    r_{\mathrm{IS}_{a_1}+\cdots+\mathrm{IS}_{a_k}}(\mathbf{x}).
\end{equation}
Consequently, suppose that $\lambda_1\geq\cdots\geq\lambda_q$ and $\mu_1\geq\cdots\geq\mu_q$ are the nonincreasing rearrangements of the multisets $\{r_{A_1+\cdots+A_k}(\mathbf{x}):\mathbf{x}\in \mathbb{Z}_2^n\}$ and
$\{r_{\mathrm{IS}_{a_1}+\cdots+\mathrm{IS}_{a_k}}(\mathbf{x}): \mathbf{x} \in \mathbb{Z}_2^n\}$, respectively. Then $$(\lambda_1,\dots,\lambda_q) \preccurlyeq (\mu_1,\dots,\mu_q).$$
\end{thm}

Our proof of Theorem \ref{thm:majorization} is based on
the compression method, which has been used in several problems in
additive combinatorics. The basic idea is to replace a set by a more
ordered set of the same size while preserving or improving the relevant
extremal quantity. Green and Tao used such methods in their work
on the Freiman--Bilu theorem~\cite{GreenTao06} and later in the
finite field setting~\cite{GreenTao09}; see also
Even-Zohar~\cite{EvenZohar12} for compressions in
$\mathbb{Z}_2^n$. The compression used here is a one-dimensional coset compression closely related to these lexicographic compressions. The
main difference lies in the quantity being optimized: rather than controlling the size of a sumset or a doubling constant, we compress several sets simultaneously and show that the sum of additive representation counts over a prescribed set does not decrease.

A key strength of our majorization theorem is that it yields a range of new results in additive combinatorics. It allows us to continue Lev's programme in the $\mathbb{Z}_2^n$ setting. In \cite{Lev01}, Theorem 1 is used to establish Theorem 2. Following the same sequence of implications, we use our main theorem to establish a functional rearrangement inequality in $\mathbb{Z}_2^n$ analogous to Theorem 2 of \cite{Lev01}. Our theorem also yields significant results in two central areas: additive $k$-tuples and sumset estimates. We discuss these below.

\subsection{Additive \texorpdfstring{\ensuremath{k}}{k}-tuples in \texorpdfstring{\ensuremath{\mathbb{Z}_2^n}}{Z2n}}
%\subsection{Additive $k$-tuples in $\mathbb{Z}_2^n$}
Additive $k$-tuples play a key role in additive combinatorics.
\begin{definition}
\label{def: additive k-tuples}
Let $(\mathbf{G},+)$ be a finite abelian group, let $k\geq 2$, and let $A_1,\dots,A_k\subseteq \mathbf{G}$. An element $(a_1,\dots,a_k)\in A_1\times\cdots\times A_k$ is called an additive $k$-tuple if $a_1+\cdots+a_{k-1}=a_k$. Let
$$
r_k(A_1,\dots,A_k)\coloneqq 
\left|\left\{(a_1,\dots,a_k)\in A_1\times\cdots\times A_k:
a_1+\cdots+a_{k-1}=a_k\right\}\right|
$$
denote the number of additive $k$-tuples in $A_1\times\cdots\times A_k$. We write $r_k(A)\coloneqq r_k(A,\dots,A)$.
\end{definition}
In $\mathbb{Z}_2^n$, the equation $a_1+\cdots+a_{k-1}=a_k$ is equivalent to $a_1+\cdots+a_k=\mathbf{0}$. Therefore, $r_k(A_1,\dots,A_k)$ counts the zero-sum $k$-tuples in $A_1\times\cdots\times A_k$. 

For \(k=3\), the solutions to
\(x_1+x_2=x_3\) are the classical \emph{Schur triples}, whose study
goes back to Schur's theorem~\cite{Sch16}; a set \(A\subseteq \mathbf{G}\) with
\(r_3(A)=0\) is called \emph{sum-free}.  Determining the largest size of a sum-free subset of a finite abelian group was a problem posed by Erd\H{o}s~\cite{E65} and later solved by Green and
Ruzsa~\cite{GR05}. A natural quantitative refinement is to determine the extremal values of \(r_3(A)\) among subsets \(A\) of prescribed cardinality. This problem was resolved by Huczynska, Mullen and Yucas~\cite{HMY09} in
\(\mathbb{Z}_p\), and by Samotij and Sudakov~\cite{SS16} for several
classes of finite abelian groups. In \cite{HJJ24}, Huczynska, Jedwab and Johnson considered the mixed count \(r_3(A,B,B)\) for
\(A,B\subseteq\mathbb{Z}_p\). Higher-order quantities $r_k$ also arise naturally. For example, in $\mathbb{Z}_2^n$, the quantity $r_4(A)$ is the additive energy of $A$, while the general quantity $r_k(A_1,\ldots,A_k)$ is a higher-order convolution count. 

Consider the following extremal problem: for fixed integers $n_1,\dots,n_k$ satisfying $1\leq n_i\leq|\mathbf{G}|$ for every $1\leq i \leq k$, determine the minimum and maximum of $r_k(A_1,\dots,A_k)$ over all subsets $A_1,\dots,A_k\subseteq\mathbf{G}$ such that
$|A_1|=n_1,\dots,|A_k|=n_k$. When $\mathbf{G}=\mathbb{Z}_p$, the corresponding maximization result follows from Lev's result on the number of solutions to linear equations, stated in the abstract of \cite{Lev01}. This, in turn, follows from Theorem 1 of the same paper. In this paper, we establish the corresponding result for affine linear equations in $\mathbb{Z}_2^n$ (Theorem~\ref{thm:linear-equations}). The following maximization result for additive $k$-tuples is an immediate consequence of Theorem~\ref{thm:linear-equations}. In this setting, initial segments of $\mathbb{Z}_2^n$ play the role of centered arithmetic progressions in $\mathbb{Z}_p$.
\begin{thm}
\label{thm:higher-order-rearrangement}
If $A_1,\ldots,A_k\subseteq\mathbb Z_2^n$, then
\begin{equation}
\label{rearrangement inequality}
        r_k(A_1,\ldots,A_k)
        \leq
        r_k(\mathrm{IS}_{|A_1|},\ldots,\mathrm{IS}_{|A_k|}).
\end{equation}
\end{thm}

Our majorization theorem has further applications to additive $k$-tuples. For odd $k$, the following result characterizes all maximizers of $r_k(A)$ over subsets $A\subseteq\mathbb{Z}_2^n$ of prescribed cardinality; the corresponding characterization of all minimizers is
given in Corollary~\ref{cor:equality-minimum}.
\begin{thm}
\label{thm:equality-maximum}
Let $k\geq 3$ be odd. Let $B\subseteq\mathbb{Z}_2^n$ have size $b>0$, and set $Q\coloneqq 2^{\lceil\log_2 b\rceil}.$
Then $B$ maximizes $r_k(X)$ among all $b$-element subsets
$X\subseteq\mathbb{Z}_2^n$ if and only if $B=K\setminus S$, where
$K\leq\mathbb{Z}_2^n$, $|K|=Q$, and $S\subseteq K$ satisfies $|S|=Q-b$ and $r_k(S)=0.$
\end{thm}

Together with Corollary~\ref{cor:equality-minimum}, this gives a complete description of the extremizers of $r_k(A)$ for odd $k$. In particular, Corollary~\ref{cor:equality-minimum} extends the minimization result of
Samotij and Sudakov~\cite{SS16} for $r_3(A)$ in $\mathbb{Z}_2^n$ to every odd $k\geq3$. Their proof for \(k=3\) combines an eigenvalue estimate of Alon and Chung~\cite{alon} with a character-theoretic analysis of Cayley graphs and an induction on the dimension, and does not appear to extend directly to higher \(k\). In contrast, our argument based on the majorization theorem applies uniformly to every odd \(k\). The restriction to odd $k$ is important for our method: for even $k$, already the case $k=4$ is closely connected to the open problem of determining the maximum size of a Sidon set in $\mathbb{Z}_2^n$; see Remark~\ref{rem:evenk}.

\subsection{Sumset estimates}
Sumset estimates form another central topic in additive combinatorics. For subsets $A_1,\ldots,A_k$ of a finite abelian group, their sumset is defined by
\begin{equation*}
    A_1+\cdots+A_k
    \coloneqq
    \{a_1+\cdots+a_k:a_1\in A_1,\dots,a_k\in A_k\}.
\end{equation*}
A fundamental problem is to obtain lower bounds for $|A_1+\dots+A_k|$ in terms of the cardinalities of the summand sets. The Cauchy--Davenport theorem~\cite{Cauchy1813,Davenport1935} states that, for any nonempty subsets $A,B\subseteq \mathbb{Z}_p$,
\begin{equation*}
    |A+B|\geq \min\{p,|A|+|B|-1\}.
\end{equation*} 

Pollard~\cite{Pollard1974} strengthened the Cauchy--Davenport theorem by considering the representation function
\begin{equation*}
    r_{A+B}(x)\coloneqq |\{(a,b)\in A\times B:a+b=x\}|
\end{equation*}
and the number of $\tau$-popular sums $$N_{\tau}(A,B)\coloneqq |\{x\in \mathbb{Z}_p:r_{A+B}(x)\geq \tau\}|.$$  Pollard proved  the following generalization of the Cauchy-Davenport theorem: for any nonempty subsets $A,B\subseteq \mathbb{Z}_p$ and every $1\leq t\leq \min\{|A|,|B|\}$,
\begin{equation*}
    \sum_{\tau=1}^t N_{\tau}(A,B)\geq t\min\{p,|A|+|B|-t\}.
\end{equation*}
In a subsequent paper \cite{Pol75}, Pollard reformulated this result in language more closely aligned with majorization, and Lev used this formulation in \cite{Lev01} to establish his Theorem 1.

For $\mathbb{Z}_2^n$, an analogue of the Cauchy--Davenport theorem is due to Bollob\'as
and Leader~\cite{BollobasLeader96}, who showed that
\[
    |A+B|
    \geq
    |\mathrm{IS}_{|A|}+\mathrm{IS}_{|B|}|
\]
for all $A,B\subseteq\mathbb{Z}_2^n$; more generally, their result
applies to abelian groups of prime-power order. We also note that, in the $\mathbb{Z}_2^n$ setting, Yuzvinsky \cite{Y81} established the lower bound $|A| \circ |B|$, where $\circ$ denotes the Hopf--Stiefel--Pfister function. 

We apply our majorization theorem to derive a
Pollard-type inequality governing the distribution of popular sums associated with higher-order representation functions in $\mathbb{Z}_2^n$. In particular, this yields the following extension of the well-known Bollob\'as--Leader inequality from two summands to an arbitrary number of summands. 
\begin{cor}
\label{cor:extension of Bollobas-Leader}
    If $A_1,\dots,A_k\subseteq \mathbb{Z}_2^n$, then
    \begin{equation*}
        |A_1+\dots+A_k| \geq |\mathrm{IS}_{|A_1|}+\dots+\mathrm{IS}_{|A_k|}|.
    \end{equation*}
\end{cor}

A noteworthy feature of our approach is that the direction of the argument is the reverse of that used in $\mathbb{Z}_p$. In $\mathbb{Z}_p$, Pollard's theorem is used to prove the corresponding majorization theorem. In $\mathbb{Z}_2^n$, by contrast, we first prove our majorization theorem (Theorem \ref{thm:majorization}) independently using compression. We then combine this result with Karamata's inequality to establish a convexity theorem, which we use to deduce a $\mathbb{Z}_2^n$ analogue of Pollard's theorem.

\subsection{Organization of the paper}
The paper is organized as follows. In Section~\ref{sec:prelim}, we introduce the necessary notation, define the relevant concepts, and record basic properties of $r_k$. In Section~\ref{sec:compression}, we develop the compression method and prove the majorization theorem, Theorem~\ref{thm:majorization}. Section~\ref{sec:app} is devoted to applications of this theorem. We first prove a precise analogue of Lev's result on linear equations, Theorem \ref{thm:linear-equations}. We then establish a convexity inequality and a Pollard-type inequality for $\mathbb{Z}_2^n$, Theorem~\ref{thm: Pollard for Z2n}, from which we derive the Bollob\'as--Leader extension, Corollary~\ref{cor:extension of Bollobas-Leader}. This section also contains results on the monotonicity of representation functions and a functional rearrangement inequality. Finally, in Section~\ref{sec:extremal}, we determine the extremal values of $r_k(A)$ for odd $k$ in Propositions \ref{cor:odd-initial-segment-value} and \ref{prop:diagonal-minimum}. We then characterize all maximizers and minimizers of $r_k(A)$ for odd $k$ in Theorem~\ref{thm:equality-maximum} and Corollary~\ref{cor:equality-minimum}, respectively.

\section{Notation and preliminaries}\label{sec:prelim}

In this section, we introduce the notation and definitions used throughout the paper.

We write $\mathbb{N}$ for the set of positive integers and $\mathbb{N}_0$ for the set of non-negative integers. For each $n\in\mathbb{N}$, let $[n]\coloneqq\{1,\dots,n\}$. Throughout the paper, for a fixed $n$, we write $q\coloneqq |\mathbb{Z}_2^n| =2^n$.

Since $\mathbb{Z}_2^n$ is the main ambient group considered in this paper, the notation $H\leq \mathbb{Z}_2^n$ indicates that $H$ is a subgroup of $\mathbb{Z}_2^n$. If $X,Y\subseteq \mathbb{Z}_2^n$ are disjoint, we write $X\sqcup Y$ for their union. We denote the complement of $X$ in $\mathbb{Z}_2^n$ by $X^c\coloneqq \mathbb{Z}_2^n\setminus X$. For $A\subseteq\mathbb{Z}_2^n$, let $\mathds{1}_A$ denote the indicator function of $A$; that is, $\mathds{1}_A(\mathbf{x})=1$ if $\mathbf{x}\in A$ and $0$ otherwise.

We identify $\mathbb{Z}_2^n$ with the set of binary strings $\mathbf{x}=x_{n-1}\dots x_0$ of length $n$ and define the \emph{binary value} function by
\begin{equation*}
\operatorname{val}:\mathbb{Z}_2^n\to \mathbb{N}_0, \qquad \operatorname{val}(\mathbf{x})\coloneqq \sum_{i=0}^{n-1}2^ix_i.
\end{equation*}
We order the elements of $\mathbb{Z}_2^n$ by increasing binary value. Accordingly, we write $$\mathbb{Z}_2^n = \{\mathbf{x}_0,\dots, \mathbf{x}_{q-1}\},$$ where $\operatorname{val}(\mathbf{x}_i)=i$ for every $0\leq i\leq q-1$. We denote the zero element of $\mathbb{Z}_2^n$ by $\mathbf{0}$ and the element whose coordinates are all equal to 1 by $\mathbf{1}$. For $i \in [n]$, define $\mathbf{e}_i\coloneqq \underbrace{0\dots0}_{n-i}1\underbrace{0\dots0}_{i-1}\in \mathbb{Z}_2^n$. Thus, $\operatorname{val}(\mathbf{e}_i)=2^{i-1}$.

\begin{definition}
Let $n\in\mathbb{N}$ and $a\in \mathbb{N}_0$ satisfy $a\leq q$. 
\begin{enumerate}[label=\arabic*)]
    \item  The \emph{initial segment} of size $a$ in $\mathbb{Z}_2^n$ is denoted by $\mathrm{IS}_a$. For $a\geq 1$, it is given by $\mathrm{IS}_a\coloneqq\{\mathbf{x}_0,\dots, \mathbf{x}_{a-1}\}$, and we set $\mathrm{IS}_0 \coloneqq \varnothing$.
    
    \item The \emph{final segment} of size $a$ in $\mathbb{Z}_2^n$ is denoted by $\mathrm{FS}_a$. For $a\geq 1$, it is given by $\mathrm{FS}_a\coloneqq\{\mathbf{x}_{q-a},\dots, \mathbf{x}_{q-1}\}$, and we set $\mathrm{FS}_0 \coloneqq \varnothing$.
\end{enumerate}
\end{definition}

For every $0\leq a \leq q$, we have $\mathrm{IS}_a \sqcup \mathrm{FS}_{q-a}=\mathbb{Z}_2^n$, and hence 
$\mathrm{IS}_a^c=\mathrm{FS}_{q-a}$. Moreover, $\mathrm{IS}_a + \mathbf{1} = \mathrm{FS}_a$. These identities will be used repeatedly. For example, when $n=3$,
 \begin{equation*}
\mathrm{IS}_4=\{000,001,010,011\},\qquad \mathrm{FS}_4=\{100,101,110,111\}.
\end{equation*}
Thus, 
\begin{equation*}
\mathrm{IS}_4 \sqcup \mathrm{FS}_4 = \mathbb{Z}_2^n \qquad \text{and} \qquad \mathrm{IS}_4 + \mathbf{1}=\mathrm{FS}_4.
\end{equation*}

Let $k\geq 2$ and $A_1,\dots,A_k\subseteq \mathbb{Z}_2^n$. For each $\mathbf{x}\in \mathbb{Z}_2^n$, define the \emph{representation function} by
\begin{equation*}
    r_{A_1+\dots+A_k}(\mathbf{x})\coloneqq |\{(a_1,\dots,a_k)\in A_1\times \dots \times A_k: a_1+\dots+a_k = \mathbf{x}\}|.
\end{equation*}
The following identity, which will be used repeatedly throughout the paper, relates the representation function $r_{A_1+\dots+A_k}$ to the quantity $r_{k+1}$ defined in Definition \ref{def: additive k-tuples}. For any $A_1,\dots,A_{k+1}\subseteq\mathbb{Z}_2^n$, we have
\begin{equation*}
r_{k+1}(A_1,\dots,A_{k+1}) = \sum_{\mathbf{x}\in A_{k+1}}r_{A_1+\dots+A_k}(\mathbf{x}).
\end{equation*}

We next record some basic but useful properties of $r_k(A_1,\dots,A_k)$, all of which follow directly from its definition. 
\begin{lemma}
\label{lem: properties of r(A,B,C)}
    Suppose that $A_1,\dots,A_k\subseteq \mathbb{Z}_2^n$. Then the following statements hold.
    \begin{enumerate}[label=\arabic*)]
        \item For every $i\in [k]$ and every $A'_i\subseteq \mathbb{Z}_2^n$ disjoint from $A_i$, $$r_k(A_1, \dots, A_i\sqcup A'_i,\dots A_k) = r_k(A_1, \dots, A_i,\dots, A_k) + r_k(A_1, \dots, A'_i,\dots, A_k).$$
        \item For every $i\in[k]$, $$r_k(A_1, \dots, A_i,\dots, A_k) + r_k(A_1, \dots, A^c_i,\dots, A_k) = \prod_{\substack{1\leq j\leq k \\ j\neq i}}|A_j|.$$
    \end{enumerate}
\end{lemma}

\section{The compression method and proof of the majorization theorem}\label{sec:compression}

In this section, we study higher-order additive tuples in $\mathbb Z_2^n$ and introduce the compression method that will serve as our main tool.

For convenience, for each $m\in\{0,1,2\}$, we denote the initial segment of $\mathbb{Z}_2$ of size $m$ by $\mathrm{I}_m$. Thus,
$$\mathrm{I}_0=\varnothing, \qquad
    \mathrm{I}_1=\{0\}, \qquad
    \mathrm{I}_2=\mathbb{Z}_2.$$

We now define the compression operation. Let \(\mathbf{v} =v_{n-1}\dots v_0\in\mathbb Z_2^n\setminus \{\mathbf{0}\} \), and define
$$j(\mathbf{v}):=\max\{0\leq j \leq n-1:v_j=1\}.$$  Let
\begin{equation*}
        L_\mathbf{v}:=\left\{\mathbf{x}\in\mathbb Z_2^n:x_{j(\mathbf{v})}=0\right\}.
\end{equation*}
It is clear that $L_\mathbf{v} \leq \mathbb{Z}_2^n$. Moreover,
\begin{equation}
\label{partition Z_2^n}
\mathbb{Z}_2^n = \bigsqcup_{\mathbf{h}\in L_\mathbf{v}} \{\mathbf{h}, \mathbf{h}+\mathbf{v}\},    
\end{equation}
and, for every $\mathbf{h}\in L_{\mathbf{v}}$,
$$\text{val}(\mathbf{h}) < \text{val}(\mathbf{h}+\mathbf{v}).$$

For \(A\subseteq\mathbb Z_2^n\), we define the $\mathbf{v}$-compression of $A$, denoted by \(\mathrm{C}_\mathbf{v}(A)\), as follows. Within each pair
$\{\mathbf{h}, \mathbf{h}+\mathbf{v}\}$, the compression preserves the number of elements of \(A\) but places them as low as possible in the ordering of \(\mathbb Z_2^n\):
\begin{equation*}
        \varnothing \mapsto\varnothing,\qquad
        \{\mathbf{h}\}\mapsto\{\mathbf{h}\},\qquad
        \{\mathbf{h}+\mathbf{v}\}\mapsto\{\mathbf{h}\},\qquad
        \{\mathbf{h},\mathbf{h}+\mathbf{v}\}\mapsto\{\mathbf{h},\mathbf{h}+\mathbf{v}\}.
\end{equation*}

More formally, since
\begin{equation*}
    A = \bigsqcup_{\mathbf{h}\in L_\mathbf{v}} A\cap \{\mathbf{h},\mathbf{h}+\mathbf{v}\} ,
\end{equation*}
we define
\begin{equation*}
    \mathrm{C}_\mathbf{v}(A) \coloneqq \bigsqcup_{\mathbf{h}\in L_\mathbf{v}} \mathrm{IS}_{|A\cap \{\mathbf{h},\mathbf{h}+\mathbf{v}\}|}(\{\mathbf{h},\mathbf{h}+\mathbf{v}\}),
\end{equation*}
where $ \mathrm{IS}_{|A\cap \{\mathbf{h},\mathbf{h}+\mathbf{v}\}|}(\{\mathbf{h},\mathbf{h}+\mathbf{v}\})$ denotes the initial segment of $\{\mathbf{h},\mathbf{h}+\mathbf{v}\}$ having cardinality $|A\cap \{\mathbf{h},\mathbf{h}+\mathbf{v}\}|$.

The $\mathbf{v}$-compression of a set $A\subseteq \mathbb{Z}_2^n$ has the following two useful properties:
\begin{enumerate}[label=\arabic*)]
    \item The cardinality of $A$ is preserved under compression; that is, $|\mathrm{C}_\mathbf{v}(A)| = |A|$.
    \item Under compression, no element of $A$ moves upward in the ordering of $\mathbb Z_2^n$ defined above.
\end{enumerate}

We next introduce the notion of a fibre.
\begin{definition}\label{def:fibre}
    Let $\mathbf{v}\in\mathbb{Z}_2^n\setminus\{\mathbf{0}\}$,
    $\mathbf{h}\in L_{\mathbf{v}}$, and
    $A\subseteq\mathbb{Z}_2^n$. The fibre of $A$ over $\mathbf{h}$
    is defined by
    \begin{equation*}
        A_{\mathbf{h}}
        \coloneqq
        \{\alpha\in\mathbb{Z}_2:
        \mathbf{h}+\alpha\mathbf{v}\in A\}.
    \end{equation*}
\end{definition}

Clearly, $A_\mathbf{h}\subseteq \mathbb{Z}_2$. Moreover, the definition of compression gives
\begin{equation}
\label{h-fibre of C_v(X)}
(\mathrm{C}_\mathbf{v}(A))_\mathbf{h} = \mathrm{I}_{|A
_\mathbf{h}|}.    
\end{equation}

\begin{lemma}
\label{lem: monotonicity of r_k for Z_2}
If $A_1,\dots,A_k\subseteq \mathbb{Z}_2$, then
\begin{equation}
\label{inequality for r_k(Z_2)}
    r_k(A_1,\dots,A_k) \leq r_k(\mathrm{I}_{|A_1|},\dots, \mathrm{I}_{|A_k|}).
\end{equation}
\end{lemma}

\begin{proof}[Proof of Lemma \ref{lem: monotonicity of r_k for Z_2}]
If $A_j=\varnothing$ for some $j\in [k]$, then $\mathrm{I}_{|A_j|} = \varnothing$, so both sides of \eqref{inequality for r_k(Z_2)} are equal to $0$.

Now suppose that none of $A_1,\ldots, A_k$ is empty. It remains to consider the following two cases:
    \begin{enumerate}[label=\arabic*)]
        \item Suppose that $A_j = \mathbb{Z}_2$ for some $j\in [k]$. Then $\mathrm{I}_{|A_j|}=\mathbb{Z}_2$, and hence $$r_k(\mathrm{I}_{|A_1|},\dots, \mathrm{I}_{|A_k|}) = \prod_{\substack{i=1 \\ i\neq j}}^{k}|\mathrm{I}_{|A_i|}|=  \prod_{\substack{i=1 \\ i\neq j}}^{k}|A_i|=r_k(A_1,\dots,A_k).$$ Thus, equality holds in this case.
        
        \item Suppose that \(A_1,\dots,A_k\) are all singletons, say $A_i = \{a_i\}$ for $i\in [k]$. Then $\mathrm{I}_{|A_i|} = \{0\}$ for every $i\in 
        [k]$, and therefore
$r_k(\mathrm{I}_{|A_1|},\dots, \mathrm{I}_{|A_k|}) = 1.$
On the other hand,
\begin{equation*}
    r_k(A_1,\dots,A_k) =
    \begin{cases}
1, & \text{if } a_1+\dots+a_k = \mathbf{0},\\
0, & \text{if } a_1+\dots+a_k \neq \mathbf{0}.
\end{cases}
\end{equation*}
Hence, the desired inequality follows.  \qedhere  
    \end{enumerate}
\end{proof}

The following lemma enables us to reduce the problem of counting additive $k$-tuples in $\mathbb{Z}_2^n$ to the corresponding problem in $\mathbb{Z}_2$.
\begin{lemma}
\label{lem: reduction from Z_2^n to Z_2}
Let $\mathbf v\in\mathbb Z_2^n\setminus\{\mathbf0\}$ and let
$A_1,\ldots,A_k\subseteq\mathbb Z_2^n$. Then
\[
r_k(A_1,\ldots,A_k)
=
\sum_{\substack{\mathbf h_1,\ldots,\mathbf h_k\in L_{\mathbf v}\\
\mathbf h_1+\cdots+\mathbf h_k=\mathbf0}}
r_k\bigl(
(A_1)_{\mathbf h_1},\ldots,(A_k)_{\mathbf h_k}
\bigr).
\]
\end{lemma}

\begin{proof}[Proof of Lemma \ref{lem: reduction from Z_2^n to Z_2}]
For $\mathbf{h}\in L_{\mathbf{v}}$, set $P^{\mathbf{h}}\coloneqq \{\mathbf{h}, \mathbf{h}+\mathbf{v}\}$. For every $i\in [k]$, we have the disjoint decomposition $A_i = \sqcup_{\mathbf{h}_i\in L_\mathbf{v}}(A_i\cap P^{\mathbf{h}_i})$. Together with part 1 of Lemma \ref{lem: properties of r(A,B,C)}, this decomposition gives
\begin{equation*}
    r_k(A_1,\dots,A_k)= \sum\limits_{\mathbf{h}_1,\dots,\mathbf{h}_k\in L_\mathbf{v}} r_k\big(A_1\cap P^{\mathbf{h}_1},\dots,A_k\cap P^{\mathbf{h}_k}\big).
\end{equation*}
We claim that if $\mathbf{h}_1,\dots,\mathbf{h}_k\in L_{\mathbf{v}}$ satisfy $\mathbf{h}_1+\dots+\mathbf{h}_k\neq \mathbf{0}$, then $r_k(A_1\cap P^{\mathbf{h}_1},\dots,A_k\cap P^{\mathbf{h}_k}) = 0$. Suppose, for contradiction, that this quantity is positive. Then there exists $(\mathbf{x}_1,\dots,\mathbf{x}_k)\in \prod_{i=1}^k (A_i\cap P^{\mathbf{h}_i})$ such that $\mathbf{x}_1+\dots+\mathbf{x}_k=\mathbf{0}$. For every $i\in [k]$, we may write $\mathbf{x}_i = \mathbf{h}_i + \alpha_i \mathbf{v}$ for some $\alpha_i \in \mathbb{Z}_2$. It follows that $$\mathbf{h}_1+\dots+\mathbf{h}_k = (\alpha_1+\dots+\alpha_k)\mathbf{v}.$$ If $\alpha_1+\dots+\alpha_k=1$, then the preceding identity implies that $\mathbf{v}\in L_{\mathbf{v}}$, a contradiction. Therefore, $\alpha_1+\dots+\alpha_k=0$, and hence $\mathbf{h}_1 + \dots + \mathbf{h}_k=\mathbf{0}$, which is also a contradiction. Thus, 
\begin{equation*}
    r_k(A_1,\dots,A_k)= \sum_{\substack{\mathbf h_1,\ldots,\mathbf h_k\in L_{\mathbf v}\\
\mathbf h_1+\cdots+\mathbf h_k=\mathbf0}} r_k\big(A_1\cap P^{\mathbf{h}_1},\dots,A_k\cap P^{\mathbf{h}_k}\big).
\end{equation*}
Finally, suppose that $\mathbf{h_1},\dots,\mathbf{h}_k\in L_\mathbf{v}$ satisfy $\mathbf{h}_1+\cdots+\mathbf{h}_k=\mathbf{0}$. The correspondence
\[
(\alpha_1,\ldots,\alpha_k)
\longmapsto
(\mathbf{h}_1+\alpha_1\mathbf{v},\ldots,
\mathbf{h}_k+\alpha_k\mathbf{v})
\]
is a bijection between the additive $k$-tuples in
$\prod_{i=1}^k (A_i)_{\mathbf{h}_i}$ and those in
$\prod_{i=1}^k (A_i\cap P^{\mathbf{h}_i})$.
Therefore,
\[
r_k\big((A_1)_{\mathbf{h}_1},\ldots,(A_k)_{\mathbf{h}_k}\big)=r_k\big(A_1\cap P^{\mathbf{h}_1},\ldots,A_k\cap P^{\mathbf{h}_k}\big),
\]
which completes the proof.
\end{proof}

We now examine how compression affects the representation function. The next lemma shows that, for any set \(B\), the sum of the representation values over \(B\) does not exceed the corresponding sum after compressing all the sets in the same direction. In particular, compression cannot decrease the number of additive \(k\)-tuples.

\begin{lemma}
\label{lem:compression-majorization-monotonicity}
Let $\mathbf{v}\in\mathbb Z_2^n\setminus\{\mathbf{0}\}$ and let
$A_1,\ldots,A_k\subseteq\mathbb Z_2^n$. Then, for every $B\subseteq\mathbb Z_2^n$,
\begin{equation}
\label{compression majorization inequality}
    \sum_{\mathbf{x}\in B} r_{A_1+\cdots+A_k}(\mathbf{x})
    \leq
    \sum_{\mathbf{x}\in \mathrm{C}_{\mathbf v}(B)} r_{\mathrm{C}_{\mathbf v}(A_1)+\cdots+\mathrm{C}_{\mathbf v}(A_k)}(\mathbf{x}).
\end{equation} In particular,
\[
    r_k(A_1,\ldots,A_k)
    \leq
    r_k(\mathrm{C}_{\mathbf v}(A_1),\ldots,\mathrm{C}_{\mathbf v}(A_k)).
\]
\end{lemma}

\begin{proof}[Proof of Lemma \ref{lem:compression-majorization-monotonicity}]
First, observe that 
\begin{equation*}
\sum_{\mathbf{x}\in B} r_{A_1+\cdots+A_k}(\mathbf{x})=r_{k+1}(A_1,\dots,A_k,B),
\end{equation*}
\begin{equation*}
\sum_{\mathbf{x}\in \mathrm{C}_{\mathbf v}(B)} r_{\mathrm{C}_{\mathbf v}(A_1)+\cdots+\mathrm{C}_{\mathbf v}(A_k)}(\mathbf{x})=r_{k+1}\big(
        \mathrm{C}_{\mathbf v}(A_1),\ldots,\mathrm{C}_{\mathbf v}(A_k),\mathrm{C}_{\mathbf v}(B)\big).
\end{equation*}
Applying Lemma \ref{lem: reduction from Z_2^n to Z_2} to the \(k+1\) sets
$A_1,\dots,A_k,B$, followed by Lemma \ref{lem: monotonicity of r_k for Z_2}, gives
$$ \begin{aligned} r_{k+1}(A_1,\dots,A_k,B) &= \sum_{\substack{\mathbf h_1,\ldots,\mathbf h_{k+1}\in L_{\mathbf v}\\
\mathbf h_1+\cdots+\mathbf h_{k+1}=\mathbf0}} r_{k+1}\left( (A_1)_{\mathbf h_1}, \dots, (A_k)_{\mathbf h_k}, B_{\mathbf h_{k+1}} \right)\\ &\leq \sum_{\substack{\mathbf h_1,\ldots,\mathbf h_{k+1}\in L_{\mathbf v}\\
\mathbf h_1+\cdots+\mathbf h_{k+1}=\mathbf0}} r_{k+1}\Big( \mathrm I_{|(A_1)_{\mathbf h_1}|}, \dots, \mathrm I_{|(A_k)_{\mathbf h_k}|}, \mathrm I_{|B_{\mathbf h_{k+1}}|} \Big). \end{aligned} $$
By \eqref{h-fibre of C_v(X)}, the last expression is equal to
$$ \sum_{\substack{\mathbf h_1,\ldots,\mathbf h_{k+1}\in L_{\mathbf v}\\
\mathbf h_1+\cdots+\mathbf h_{k+1}=\mathbf0}} r_{k+1}\left( \bigl(\mathrm C_{\mathbf v}(A_1)\bigr)_{\mathbf h_1}, \dots, \bigl(\mathrm C_{\mathbf v}(A_k)\bigr)_{\mathbf h_k}, \bigl(\mathrm C_{\mathbf v}(B)\bigr)_{\mathbf h_{k+1}} \right). $$
Another application of Lemma \ref{lem: reduction from Z_2^n to Z_2} shows that this sum is equal to
$$ r_{k+1}\big( \mathrm C_{\mathbf v}(A_1), \dots, \mathrm C_{\mathbf v}(A_k), \mathrm C_{\mathbf v}(B) \big).$$
The two identities at the beginning of the proof now establish \eqref{compression majorization inequality}. Finally,
taking $B=\{\mathbf{0}\}$ in \eqref{compression majorization inequality} and observing that
$\mathrm{C}_{\mathbf v}(B)=\{\mathbf{0}\}$ yields the final assertion.
\end{proof}

We now characterize the subsets of $\mathbb{Z}_2^n$ that are fixed by every compression.

\begin{lemma}
\label{lem:compression-sets-initial}
Let \(A\subseteq\mathbb Z_2^n\).  If \(\mathrm{C}_{\mathbf{v}}(A)=A\) for every ${\mathbf{v}}\in\mathbb Z_2^n \setminus \{\mathbf{0}\}$, then \(A=\mathrm{IS}_{|A|}\).
\end{lemma}

\begin{proof}[Proof of Lemma \ref{lem:compression-sets-initial}]
Suppose that \(A\) is not an initial segment of $\mathbb{Z}_2^n$.  Then $A$ has a gap; that is, there exist 
$\mathbf{x}\in A$ and $\mathbf{y}\notin A$ such that $\operatorname{val}(\mathbf{y}) < \operatorname{val}(\mathbf{x}).$  Set $\mathbf{v}\coloneqq \mathbf{x}+\mathbf{y}.$ Since $\mathbf{x} \neq \mathbf{y}$, we have $\mathbf{v}\neq \mathbf{0}$. Let \(\ell\) be
the largest index at which \(\mathbf{x}\) and \(\mathbf{y}\) differ.  The inequality $\operatorname{val}(\mathbf{y}) < \operatorname{val}(\mathbf{x})$ implies that \(y_{\ell}=0\) and \(x_{\ell}=1\), while \(x_i=y_i\) for every \(i>\ell\).  Hence, \(\ell=j(\mathbf{v})\), so \(\mathbf{y}\in L_\mathbf{v}\) and \(\mathbf{x}=\mathbf{y}+\mathbf{v}\). Thus, the pair \(\{\mathbf{y},\mathbf{y}+\mathbf{v}\}\) is precisely \(\{\mathbf{y},\mathbf{x}\}\). Since $A$ contains the upper element $\mathbf{x}$ but not the lower element $\mathbf{y}$, the $\mathbf{v}$-compression changes $A$, contradicting the assumption. Therefore, $A$ has no gaps and hence $A=\mathrm{IS}_{|A|}$.
\end{proof}

We are now ready to prove Theorem \ref{thm:majorization}, our main rearrangement result. It states that the number of $k$-tuples from \(A_1\times\dots\times A_k\) whose sum lies in \(B\) is maximized when \(A_1,\dots,A_k\) and \(B\) are replaced by initial segments of the same cardinalities.

\iffalse
\begin{thm}\label{thm:majorization}
Let $0\leq a_1,\dots,a_k\leq q$, and let $A_1,\dots,A_k\subseteq \mathbb{Z}_2^n$ satisfy $|A_i|=a_i$ for every $i\in [k]$. Then, for every
$B\subseteq\mathbb Z_2^n$, writing $b\coloneqq |B|$, we have
\begin{equation}
\label{eq:setwise-rearrangement}
    \sum_{\mathbf{x}\in B}
    r_{A_1+\cdots+A_k}(\mathbf{x})
    \leq
    \sum_{\mathbf{x}\in\mathrm{IS}_{b}}
    r_{\mathrm{IS}_{a_1}+\cdots+\mathrm{IS}_{a_k}}(\mathbf{x}).
\end{equation}
Consequently, suppose that $\lambda_1\geq\cdots\geq\lambda_q$ and $\mu_1\geq\cdots\geq\mu_q$ are the nonincreasing rearrangements of the multisets $\{r_{A_1+\cdots+A_k}(\mathbf{x}):\mathbf{x}\in \mathbb{Z}_2^n\}$ and
$\{r_{\mathrm{IS}_{a_1}+\cdots+\mathrm{IS}_{a_k}}(\mathbf{x}): \mathbf{x} \in \mathbb{Z}_2^n\}$, respectively. Then $$(\lambda_1,\dots,\lambda_q) \preccurlyeq (\mu_1,\dots,\mu_q).$$
\end{thm}
\fi

\begin{proof}[Proof of Theorem \ref{thm:majorization}]
For every $(k+1)$-tuple $(X_1,\dots,X_{k+1})$ of subsets of $\mathbb{Z}_2^n$, define its \emph{weight} by
\begin{equation*}
    \mathrm{W}(X_1,\dots,X_{k+1}) \coloneqq \sum_{i=1}^{k+1} \sum_{\mathbf{x}\in X_i}\text{val}(\mathbf{x}). 
\end{equation*}
Clearly, $\mathrm{W}(X_1,\dots,X_{k+1})\in\mathbb{N}_0$. For every $(k+1)$-tuple $(X_1,\dots,X_{k+1})$, exactly one of the following two alternatives holds.
\begin{enumerate}[label=\arabic*)]
    \item For every $\mathbf{v}\in \mathbb{Z}_2^n \setminus \{\mathbf{0}\}$, $
        (\mathrm{C}_\mathbf{v}(X_1),\dots,\mathrm{C}_\mathbf{v}(X_{k+1})) = (X_1,\dots,X_{k+1}). $
    In this case, Lemma \ref{lem:compression-sets-initial} gives, $X_i = \mathrm{IS}_{|X_i|}$ for every $i\in [k+1]$.

    \item There exists $\mathbf{v}\in \mathbb{Z}_2^n \setminus \{\mathbf{0}\}$ such that $(\mathrm{C}_\mathbf{v}(X_1),\dots,\mathrm{C}_\mathbf{v}(X_{k+1})) \neq (X_1,\dots,X_{k+1}). $
    Thus, $\mathrm{C}_\mathbf{v}(X_j)\neq X_j$ for some $j\in [k+1]$. Since compression does not move any element upward, at least one element of \(X_j\) must move strictly downward. Therefore,
    \begin{equation*}
        \mathrm{W}(X_1,\dots,X_{k+1}) > \mathrm{W}(\mathrm{C}_\mathbf{v}(X_1),\dots,\mathrm{C}_\mathbf{v}(X_{k+1})). 
    \end{equation*}
\end{enumerate}
Furthermore, \eqref{compression majorization inequality} gives 
\begin{equation*}
        \sum_{\mathbf{x}\in X_{k+1}} r_{X_1+\dots+X_k}(\mathbf{x}) \leq \sum_{\mathbf{x}\in \mathrm{C}_{\mathbf{v}}(X_{k+1})} r_{\mathrm{C}_{\mathbf{v}}(X_1)+\dots+\mathrm{C}_{\mathbf{v}}(X_k)}(\mathbf{x}).
\end{equation*}

We begin with the $(k+1)$-tuple $(A_1,\dots,A_k,B)$. If the first alternative holds, then every set is already an initial segment, and \eqref{eq:setwise-rearrangement} holds with equality. If the second alternative holds, we replace $(A_1,\dots,A_k,B)$ by $(\mathrm{C}_\mathbf{v}(A_1),\dots,\mathrm{C}_\mathbf{v}(A_k),\mathrm{C}_\mathbf{v}(B))$. The preceding inequality shows that this replacement does not decrease the relevant sum, while the weight strictly decreases. We then apply the same procedure to the resulting $(k+1)$-tuple.

This procedure must terminate because the weight is a non-negative integer and decreases strictly at every nontrivial step. Thus, after finitely many steps, we obtain a \((k+1)\)-tuple $(A_1^\ast,\dots,A_k^\ast,B^\ast)$ that is fixed under every \(\mathbf v\)-compression. By Lemma \ref{lem:compression-sets-initial}, $A_i^\ast=\mathrm{IS}_{a_i}$ for every $i\in[k]$, and $ B^\ast=\mathrm{IS}_b$. It follows that
\begin{equation*}
    \sum_{\mathbf{x}\in B} r_{A_1+\dots+A_k}(\mathbf{x}) \leq \sum_{\mathbf{x}\in \mathrm{IS}_b} r_{\mathrm{IS}_{a_1}+\dots+\mathrm{IS}_{a_k}}(\mathbf{x}).
\end{equation*}

Now fix $\ell \in [q]$ and choose an $\ell$-element set $B_\ell$ on
which $r_{A_1+\cdots+A_k}$ takes its $\ell$ largest values. By \eqref{eq:setwise-rearrangement},
\[
    \sum_{i=1}^{\ell}\lambda_i
    =
    \sum_{\mathbf{x}\in B_\ell}r_{A_1+\cdots+A_k}(\mathbf{x})
    \leq
    \sum_{\mathbf{x}\in\mathrm{IS}_{\ell}}
    r_{\mathrm{IS}_{a_1}+\cdots+\mathrm{IS}_{a_k}}(\mathbf{x})
    \leq
    \sum_{i=1}^{\ell}\mu_i.
\]
For $\ell=q$, both sides are equal to $\prod_{i=1}^k a_i$, which proves the majorization statement. 
\end{proof}

\section{Applications of the majorization theorem}\label{sec:app}
In this section, we present several applications and consequences of the majorization theorem.

\subsection{Linear equations and extremal additive tuples}
We begin with an application to linear equations over $\mathbb Z_2^n$ which is a precise analogue of the first result presented in the abstract of \cite{Lev01}. In this setting, invertible linear transformations play the same role as nonzero scalar coefficients in Lev's theorem. 

Let $0\leq a_1,\dots,a_k\leq q$, and let $A_1,\dots,A_k\subseteq \mathbb{Z}_2^n$ satisfy $|A_i|=a_i$ for every $i\in [k]$. Let $$\mathrm{T}_1,\ldots,\mathrm{T}_k\in\mathrm{GL}_n(\mathbb Z_2).$$ For each $\mathbf{x}\in \mathbb{Z}_2^n$, define $$N(\mathbf{x})\coloneqq |\{(\mathbf{x}_1,\dots,\mathbf{x}_k)\in A_1\times \dots\times A_k: \mathrm{T}_1\mathbf{x}_1+\cdots+\mathrm{T}_k\mathbf{x}_k = \mathbf{x}\}|.$$

Theorem~\ref{thm:majorization} yields the following rearrangement inequality.

\begin{thm}
\label{thm:linear-equations}
With the notation above, for every $B\subseteq \mathbb{Z}_2^n$, writing $b\coloneqq |B|$, we have
\begin{equation}
\label{majorization for N(x)}
    \sum_{\mathbf{x}\in B}N(\mathbf{x}) \leq \sum_{\mathbf{x}\in \mathrm{IS}_b}r_{\mathrm{IS}_{a_1}+\dots+\mathrm{IS}_{a_k}}(\mathbf{x}).
\end{equation}
In particular, for every $\mathbf{x}\in \mathbb{Z}_2^n$, $$N(\mathbf{x})\leq r_k(\mathrm{IS}_{a_1},\dots,\mathrm{IS}_{a_k}).$$
\end{thm}

\begin{proof}[Proof of Theorem \ref{thm:linear-equations}]
For every $i\in [k]$, set $B_i\coloneqq \mathrm{T}_i(A_i)$. Since $\mathrm{T}_i\in \mathrm{GL}_n(\mathbb{Z}_2)$, the map $\mathrm{T}_i$ is a bijection. Therefore, $|B_i|=|A_i|=a_i$. Moreover, the correspondence $$(\mathbf{x}_1,\dots,\mathbf{x}_k)\longmapsto (\mathrm{T}_1\mathbf{x}_1,\dots,\mathrm{T}_k\mathbf{x}_k)$$ is a bijection from $\prod_{i=1}^k A_i$ onto $\prod_{i=1}^k B_i$. Hence, for every $\mathbf{x}\in \mathbb{Z}_2^n$, $N(\mathbf{x}) = r_{B_1+\dots+B_k}(\mathbf{x}).$
Applying Theorem \ref{thm:majorization} to the sets $B_1,\dots,B_k$ gives
\begin{equation*}
    \sum_{\mathbf{x}\in B}N(\mathbf{x}) = \sum_{\mathbf{x}\in B}r_{B_1+\dots+B_k}(\mathbf{x})\leq \sum_{\mathbf{x}\in \mathrm{IS}_b}r_{\mathrm{IS}_{|B_1|}+\dots+\mathrm{IS}_{|B_k|}}(\mathbf{x}) = \sum_{\mathbf{x}\in \mathrm{IS}_b}r_{\mathrm{IS}_{a_1}+\dots+\mathrm{IS}_{a_k}}(\mathbf{x}),
\end{equation*}
which proves \eqref{majorization for N(x)}. Finally, take $B=\{\mathbf{x}\}$ in \eqref{majorization for N(x)}. Since $b=1$ and $\mathrm{IS}_1 = \{\mathbf{0}\}$, we obtain
$$N(\mathbf{x})\leq r_{\mathrm{IS}_{a_1}+\dots+\mathrm{IS}_{a_k}}(\mathbf{0})=r_k(\mathrm{IS}_{a_1},\dots,\mathrm{IS}_{a_k}),$$
which proves the final assertion.
\end{proof}

Note that Theorem~\ref{thm:higher-order-rearrangement} is a special case of Theorem~\ref{thm:linear-equations}, obtained by taking $\mathrm{T}_i=\operatorname{id}$ for every $i\in [k]$. Indeed, in this case, $N= r_{A_1+\dots+A_k}$. In particular, $N(\mathbf{0}) = r_k(A_1,\dots,A_k)$.

Next, we determine the minimum possible value of \(r_k(A_1,\dots,A_k)\) when the sets \(A_1,\dots,A_k\subseteq \mathbb Z_2^n\) vary independently subject to prescribed cardinalities.

\begin{cor}
\label{cor:mixed-k-minimum}
Let $0\leq a_1,\dots,a_k\leq q$. Then
\begin{equation*}
\min_{\substack{A_i\subseteq \mathbb{Z}_2^n\\ |A_i|=a_i}}
        r_k(A_1,\ldots,A_k)=r_k(\mathrm{IS}_{a_1},\dots,\mathrm{IS}_{a_{k-1}},\mathrm{FS}_{a_k}).
\end{equation*}
\end{cor}
\begin{proof}[Proof of Corollary \ref{cor:mixed-k-minimum}]
By part 2 of Lemma~\ref{lem: properties of r(A,B,C)} and
Theorem~\ref{thm:higher-order-rearrangement}, we have
\begin{align*}
r_k(A_1,\ldots,A_k)
&=\prod_{i=1}^{k-1}a_i-r_k(A_1,\ldots,A_{k-1},A_k^c)\\
&\geq \prod_{i=1}^{k-1}a_i-
r_k(\mathrm{IS}_{a_1},\ldots,\mathrm{IS}_{a_{k-1}},
\mathrm{IS}_{q-a_k})=r_k(\mathrm{IS}_{a_1},\ldots,\mathrm{IS}_{a_{k-1}},
\mathrm{FS}_{a_k}).
\end{align*}
Equality is attained by the tuple $(\mathrm{IS}_{a_1},\ldots,\mathrm{IS}_{a_{k-1}},\mathrm{FS}_{a_k})$.
\end{proof}

\subsection{Convexity and Pollard-type inequalities in \texorpdfstring{\ensuremath{\mathbb{Z}_2^n}}{Z2n}}
In this subsection, we derive a general convexity inequality from Karamata's inequality~\cite{Kar32} and use it to obtain Pollard-type and sumset inequalities. 

Together, the majorization theorem (Theorem~\ref{thm:majorization}) and Karamata's inequality (Proposition~\ref{prop:karam}) immediately yield the following \emph{convexity inequality}.

\begin{thm}
\label{thm: main convexity inequality}
Let $0\leq a_1,\dots,a_k\leq q$, and let $A_1,\dots,A_k\subseteq \mathbb{Z}_2^n$ satisfy $|A_i|=a_i$ for every $i\in [k]$. If $\Phi:[0,\infty)\to \mathbb{R}$ is convex, then 
\begin{equation}
\label{main convexity inequality}
\sum_{\mathbf{x}\in \mathbb{Z}_2^n} \Phi(r_{A_1+\dots+A_k}(\mathbf{x})) \leq \sum_{\mathbf{x}\in \mathbb{Z}_2^n} \Phi(r_{\mathrm{IS}_{a_1}+\dots+\mathrm{IS}_{a_k}}(\mathbf{x})).
\end{equation}
\end{thm}

\begin{proof}[Proof of Theorem \ref{thm: main convexity inequality}]
By Theorem~\ref{thm:majorization}, the nonincreasing rearrangement of the values of $r_{A_1+\cdots+A_k}$ is majorized by the nonincreasing rearrangement of the values of $r_{\mathrm{IS}{a_1}+\cdots+\mathrm{IS}{a_k}}$. The result therefore follows immediately from Karamata's inequality.
\end{proof}

We next use Theorem~\ref{thm: main convexity inequality} to derive a Pollard-type inequality in $\mathbb{Z}_2^n$. For $\tau\in \mathbb{N}$ and $X_1,\dots,X_k\subseteq\mathbb{Z}_2^n$, define the set of \emph{$\tau$-popular sums} by
\begin{equation*}
S_{\tau}(X_1,\dots,X_k)
\coloneqq
\left\{\mathbf{x}\in\mathbb{Z}_2^n:
r_{X_1+\dots+X_k}(\mathbf{x})\geq\tau\right\}.
\end{equation*}

\begin{thm}
\label{thm: Pollard for Z2n}
Let $0\leq a_1,\dots,a_k\leq q$, and let $A_1,\dots,A_k\subseteq \mathbb{Z}_2^n$ satisfy $|A_i|=a_i$ for every $i\in [k]$. Then, for every $N\in \mathbb{N}$,
\begin{equation*}
    \sum_{\tau=1}^N |S_{\tau}(A_1,\dots,A_k)| \geq \sum_{\tau=1}^N |S_{\tau}(\mathrm{IS}_{a_1},\dots,\mathrm{IS}_{a_k})|.
\end{equation*}
\end{thm}

\begin{proof}[Proof of Theorem \ref{thm: Pollard for Z2n}]
By the definition of $S_{\tau}$ and changing the order of summation, we obtain
\begin{equation*}
\sum_{\tau=1}^N |S_{\tau}(A_1,\dots,A_k)| = \sum_{\mathbf{x}\in \mathbb{Z}_2^n} \min\{r_{A_1+\dots+A_k}(\mathbf{x}), N\}.
\end{equation*}
Similarly,
\begin{equation*}
\sum_{\tau=1}^N |S_{\tau}(\mathrm{IS}_{a_1},\dots,\mathrm{IS}_{a_k})| = \sum_{\mathbf{x}\in \mathbb{Z}_2^n} \min\{r_{\mathrm{IS}_{a_1}+\dots+\mathrm{IS}_{a_k}}(\mathbf{x}), N\}.
\end{equation*}
For every $c>0$, define $\Phi_c:[0,\infty)\to \mathbb{R}$ by $\Phi_c(x)\coloneqq -\min\{x,c\}$. Since $\Phi_c$ is convex, the convexity inequality gives
\begin{equation}
\sum_{\mathbf{x}\in \mathbb{Z}_2^n} \min\{r_{A_1+\dots+A_k}(\mathbf{x}), c\} \geq \sum_{\mathbf{x}\in \mathbb{Z}_2^n} \min\{r_{\mathrm{IS}_{a_1}+\dots+\mathrm{IS}_{a_k}}(\mathbf{x}),c\}
\end{equation}
for every $c>0$. Taking $c=N$ yields the desired result.
\end{proof}

Taking $N=1$ in Theorem~\ref{thm: Pollard for Z2n} immediately yields Corollary~\ref{cor:extension of Bollobas-Leader}, which extends the Bollob\'as--Leader sumset inequality to an arbitrary number of summands in $\mathbb{Z}_2^n$.

\begin{remark}
We next present a further application of the convexity inequality, this time in a graph-theoretic setting. Fix $m\geq2$, and let $\Phi_m:[0,\infty)\to[0,\infty)$ be the
piecewise-linear interpolation of $j\mapsto\binom{j}{m}$ on
$\mathbb N_0$. Since
$\binom{j+1}{m}-\binom{j}{m}=\binom{j}{m-1}$ is nondecreasing in $j$, the function $\Phi_m$ is convex. Hence
Theorem~\ref{thm: main convexity inequality} gives
\[
    \sum_{\mathbf{x}\in\mathbb Z_2^n}
    \binom{r_{A_1+\cdots+A_k}(\mathbf{x})}{m}
    \leq
    \sum_{\mathbf{x}\in\mathbb Z_2^n}
    \binom{r_{\mathrm{IS}_{a_1}+\cdots+\mathrm{IS}_{a_k}}(\mathbf{x})}{m}.
\]

This inequality has a simple graph-theoretic interpretation. Let
$G(A_1,\ldots,A_k)$ be the graph with vertex set
$A_1\times\cdots\times A_k$, in which two distinct vertices are adjacent if and only if the sums of their coordinates are equal. Then
\[
    G(A_1,\ldots,A_k)
    \cong
    \bigsqcup_{\mathbf{x}\in\mathbb Z_2^n}
    K_{r_{A_1+\cdots+A_k}(\mathbf{x})}.
\]
Thus, the left-hand side of the inequality above counts the number of $m$-cliques in
$G(A_1,\ldots,A_k)$. Consequently, among all choices of sets with prescribed
cardinalities, this number is maximized when
$A_i=\mathrm{IS}_{a_i}$ for every $i\in[k]$.
\end{remark}

\subsection{Monotonicity of representation functions for initial and final segments}
The second assertion of Lev’s Theorem 1 describes the ordering of the representation numbers for the extremizing balanced sets. We now establish the corresponding monotonicity property for initial segments of \(\mathbb Z_2^n\).

Recall that $\mathbb{Z}_2^n = \{\mathbf{x}_0,\dots, \mathbf{x}_{q-1}\}$, where $\operatorname{val}(\mathbf{x}_i) = i$ for every $0\leq i \leq q-1$.

\begin{prop}
\label{prop: monotonicity of repr on IS}
Let $0\leq a_1,\dots,a_k\leq q$, and set $r\coloneqq r_{\mathrm{IS}_{a_1}+\dots+\mathrm{IS}_{a_k}}$. Then 
\begin{equation*}
    r(\mathbf{x}_0) \geq r(\mathbf{x}_1) \geq \dots \geq r(\mathbf{x}_{q-1}).
\end{equation*}
\end{prop}

\begin{proof}[Proof of Proposition \ref{prop: monotonicity of repr on IS}] 
Fix $0\leq t \leq q-2$, and set $C_t\coloneqq \mathrm{IS}_t \sqcup \{\mathbf{x}_{t+1}\}$. Then $|C_t|=t+1$. By Theorem~\ref{thm:higher-order-rearrangement},
\begin{equation*}
    r_{k+1}(\mathrm{IS}_{a_1},\dots, \mathrm{IS}_{a_k},C_t) \leq r_{k+1}(\mathrm{IS}_{a_1},\dots, \mathrm{IS}_{a_k},\mathrm{IS}_{t+1}).
\end{equation*}
Equivalently, \begin{equation*}
    \sum_{\mathbf{x}\in C_t} r(\mathbf{x}) \leq \sum_{\mathbf{x}\in \mathrm{IS}_{t+1}} r(\mathbf{x}).
\end{equation*}
Since $\mathrm{IS}_{t+1} = \mathrm{IS}_t \sqcup \{\mathbf{x}_t\}$, cancelling the common contribution from \(\mathrm{IS}_t\) gives
\begin{equation*}
    r(\mathbf{x}_{t+1}) \leq r(\mathbf{x}_{t}). 
\end{equation*}
Since this holds for every $t\in\{0,\dots,q-2\}$, the desired monotonicity follows.
\end{proof}

We next consider the more general situation in which each set is either an
initial segment or a final segment of $\mathbb{Z}_2^n$.

\begin{prop}
\label{prop: either IS or FS}
Let $0\leq a_1,\dots,a_k \leq q$, and let $B_i \in \{\mathrm{IS}_{a_i}, \mathrm{FS}_{a_i}\}$ for every $i\in [k]$. Let $I\coloneqq \{i\in [k]: B_i = \mathrm{FS}_{a_i}\}$, and let $h\coloneqq |I|$ denote the number of final segments among $B_1,\dots,B_k$. Set $r \coloneqq r_{B_1+\dots+B_k}$. Then the following hold:
\begin{enumerate}[label=\arabic*)]
    \item If $h$ is even, then $$r(\mathbf{x}_0)\geq r(\mathbf{x}_1) \geq \dots \geq r(\mathbf{x}_{q-1}).$$ 
    \item If $h$ is odd, then $$r(\mathbf{x}_0)\leq r(\mathbf{x}_1) \leq \dots \leq r(\mathbf{x}_{q-1}).$$ 
\end{enumerate}
\end{prop}

\begin{proof}[Proof of Proposition \ref{prop: either IS or FS}] Recall that $\mathrm{FS}_a = \mathrm{IS}_a + \mathbf{1}$. Thus, for each $i\in[k]$, every $b_i\in B_i$ can be written uniquely in the
form
\[
b_i=
\begin{cases}
y_i, & i\notin I,\\
\mathbf{1}+y_i, & i\in I,
\end{cases}
\qquad
y_i\in\mathrm{IS}_{a_i}.
\]
Hence the correspondence $(b_1,\dots,b_k) \longmapsto (y_1,\dots,y_k)$ is a bijection between the representations of $\mathbf{x}$ by  $B_1,\dots,B_k$ and the representations of $\mathbf{x}+h\mathbf{1}$ by $\mathrm{IS}_{a_1},\dots,\mathrm{IS}_{a_k}$. Therefore,
\begin{equation*}
    r(\mathbf{x}) = r_{\mathrm{IS}_{a_1}+\dots+\mathrm{IS}_{a_k}}(\mathbf{x} + h\mathbf{1}).
\end{equation*}
We now distinguish two cases:

1) If $h$ is even, then $h\mathbf{1} = \mathbf{0}$, and hence $r(\mathbf{x}) = r_{\mathrm{IS}_{a_1}+\dots+\mathrm{IS}_{a_k}}(\mathbf{x})$. The first conclusion therefore follows from
Proposition~\ref{prop: monotonicity of repr on IS}.

2) If $h$ is odd, then $h\mathbf{1} = \mathbf{1}$, and hence $r(\mathbf{x}) = r_{\mathrm{IS}_{a_1}+\dots+\mathrm{IS}_{a_k}}(\mathbf{x}+\mathbf{1})$. Since $\mathbf{x}_i + \mathbf{1} = \mathbf{x}_{q-1-i}$ for every $i\in \{0,\dots,q-1\}$, Proposition \ref{prop: monotonicity of repr on IS} implies $r(\mathbf{x}_0)\leq  r(\mathbf{x}_1)\leq\dots\leq r(\mathbf{x}_{q-1})$.
\end{proof}

In particular, when all the sets are final segments, we have $h=k$. Thus, the representation function $r_{\mathrm{FS}_{a_1}+\dots+\mathrm{FS}_{a_k}}$ is decreasing if $k$ is even and increasing if $k$ is odd.

% The following corollary follows immediately from the preceding proposition and concerns the case in which all the sets are final segments.

% \begin{cor}
% Let $0\leq a_1,\dots,a_k \leq q$, and set $r \coloneqq r_{\mathrm{FS}_{a_1}+\dots+\mathrm{FS}_{a_k}}$. Then the following hold:
% \begin{enumerate}[label=\arabic*)]
%     \item If $k$ is even, then $$r(\mathbf{x}_0)\geq r(\mathbf{x}_1) \geq \dots \geq r(\mathbf{x}_{q-1}).$$ 
%     \item If $k$ is odd, then $$r(\mathbf{x}_0)\leq r(\mathbf{x}_1) \leq \dots \leq r(\mathbf{x}_{q-1}).$$ 
% \end{enumerate}
% \end{cor}

\subsection{A functional rearrangement inequality} We now extend our rearrangement inequality \eqref{eq:setwise-rearrangement} for subsets of $\mathbb{Z}_2^n$ to arbitrary nonnegative functions on \(\mathbb Z_2^n\).

\begin{definition}
\label{def: decreasing rearrangement}
    Let $f:\mathbb{Z}_2^n \to [0,\infty)$, and let $a_0\geq a_1 \geq \dots \geq a_{q-1}$ be the values of $f$, counted with multiplicity and listed in nonincreasing order. The nonincreasing rearrangement of $f$ is the function $\widebar{f}:\mathbb{Z}_2^n \to [0,\infty)$ defined by $\widebar{f}(\mathbf{x}_i)\coloneqq a_i$ for every $i\in \{0,1,\dots,q-1\}$.
\end{definition}

By Definition~\ref{def: decreasing rearrangement}, the nonincreasing rearrangement of $\mathbf{1}_{A}$ is $\mathbf{1}_{\mathrm{IS}_{|A|}}$, where $A\subseteq \mathbb{Z}_2^n$. Thus, the following theorem generalizes \eqref{eq:setwise-rearrangement} (by taking $f_i=\mathbf{1}_{A_i}$ for $i\in[k]$ and $f_{k+1}=\mathbf{1}_B$) and provides an analogue of Lev's Theorem 2 in $\mathbb{Z}_2^n$.
\begin{thm}
\label{thm: Lev's thm 2 in Z2n}
    Let $k\geq 2$, and let $f_1,\dots,f_k:\mathbb{Z}_2^n\to [0,\infty)$. Then
    \begin{equation}
    \label{inequality for convolutions}
        \sum_{\substack{\mathbf{y}_1,\dots,\mathbf{y}_k\in \mathbb{Z}_2^n \\ \mathbf{y}_1+\dots+\mathbf{y}_k=\mathbf{0}}} f_1(\mathbf{y}_1)\dots f_k(\mathbf{y}_k)\leq \sum_{\substack{\mathbf{y}_1,\dots,\mathbf{y}_k\in \mathbb{Z}_2^n \\ \mathbf{y}_1+\dots+\mathbf{y}_k=\mathbf{0}}} \widebar{f_1}(\mathbf{y}_1)\dots \widebar{f_k}(\mathbf{y}_k).
    \end{equation}
\end{thm}

\begin{proof}[Proof of Theorem \ref{thm: Lev's thm 2 in Z2n}]
For each $i\in[k]$, choose an ordering $\mathbf{y}_{i,1},\ldots,\mathbf{y}_{i,q}$ of the elements of $\mathbb{Z}_2^n$ such that $f_i(\mathbf{y}_{i,1})\geq \cdots \geq f_i(\mathbf{y}_{i,q}).$
Set $a_{i,j}:=f_i(\mathbf{y}_{i,j})$ for every $j\in [q]$ and $a_{i,q+1}:=0$.  For each $j\in [q]$, let
$\mathrm{Y}_{i,j}:=\{\mathbf{y}_{i,1},\ldots,\mathbf{y}_{i,j}\}$. Thus, $\mathrm{Y}_{i,j}$ is a $j$-element set on which $f_i$ takes its $j$ largest values. We have
\[
f_i=\sum_{j=1}^q (a_{i,j}-a_{i,j+1})\mathbf 1_{\mathrm{Y}_{i,j}}
\qquad \text{and} \qquad
\widebar f_i=\sum_{j=1}^q (a_{i,j}-a_{i,j+1})
\mathbf 1_{\operatorname{IS}_j}.
\]
Indeed, if $\mathbf{x}\in \mathbb{Z}_2^n$, then $\mathbf{x}=\mathbf{y}_{i,\ell}$ for some $\ell\in [q]$. Moreover, $\mathbf{y}_{i,\ell} \in \mathrm{Y}_{i,j}$ $\Leftrightarrow$ $\ell\leq j\leq q$. Therefore, 
\begin{equation*}
\sum_{j=1}^q (a_{i,j}-a_{i,j+1})\mathbf 1_{\mathrm{Y}_{i,j}}(\mathbf{x})=\sum_{j=\ell}^q (a_{i,j}-a_{i,j+1})=a_{i,\ell}=f_i(\mathbf{x}).    
\end{equation*}
The second identity follows in the same way from the definition of the nonincreasing rearrangement. Using these decompositions and applying Theorem~\ref{thm:higher-order-rearrangement} to each $(j_1,\ldots,j_k)\in [q]^k$, we obtain
\begin{align*}
\sum_{\substack{\mathbf{y}_1,\ldots,\mathbf{y}_k\in\mathbb Z_2^n\\
\mathbf{y}_1+\cdots+\mathbf{y}_k=0}}
f_1(\mathbf{y}_1)\cdots f_k(\mathbf{y}_k)
&=
\sum_{j_1,\ldots,j_k\in [q]}
\left(\prod_{i=1}^k(a_{i,j_i}-a_{i,j_i+1})\right)
r_k(\mathrm{Y}_{1,j_1},\ldots,\mathrm{Y}_{k,j_k})\\
&\leq
\sum_{j_1,\ldots,j_k\in [q]}
\left(\prod_{i=1}^k(a_{i,j_i}-a_{i,j_i+1})\right)
r_k(\operatorname{IS}_{j_1},\ldots,\operatorname{IS}_{j_k})\\
&=
\sum_{\substack{\mathbf{y}_1,\ldots,\mathbf{y}_k\in\mathbb Z_2^n\\
\mathbf{y}_1+\cdots+\mathbf{y}_k=0}}
\widebar f_1(\mathbf{y}_1)\cdots \widebar f_k(\mathbf{y}_k)
\end{align*}
as required, 
where the inequality follows because $a_{i,j}-a_{i,j+1}\geq 0$ for all $(i,j)\in [k]\times [q]$.
\end{proof}

\begin{remark}
The analogue of Lev’s Theorem 3 in the setting of $\mathbb{Z}_2^n$ follows immediately from Theorem \ref{thm: Lev's thm 2 in Z2n}. Indeed, for functions
$f_1,\ldots,f_k:\mathbb Z_2^n \to [0,\infty)$, Theorem \ref{thm: Lev's thm 2 in Z2n}, applied to the $2k$ functions $f_1,\ldots,f_k,f_1,\ldots,f_k,$ 
gives
\[
\sum_{\substack{\mathbf{x}_1+\cdots+\mathbf{x}_k=\mathbf{y}_1+\cdots+\mathbf{y}_k}}
    \prod_{i=1}^k f_i(\mathbf{x}_i)f_i(\mathbf{y}_i)
\leq
\sum_{\substack{\mathbf{x}_1+\cdots+\mathbf{x}_k=\mathbf{y}_1+\cdots+\mathbf{y}_k}}
    \prod_{i=1}^k \widebar f_i(\mathbf{x}_i)\widebar f_i(\mathbf{y}_i),
\]
since in $\mathbb Z_2^n$ the relation
$\mathbf{x}_1+\cdots+\mathbf{x}_k=\mathbf{y}_1+\cdots+\mathbf{y}_k$ is equivalent to
$\mathbf{x}_1+\cdots+\mathbf{x}_k+\mathbf{y}_1+\cdots+\mathbf{y}_k=\mathbf{0}$.
Thus, no separate argument corresponding to Lev's Theorem~3 is required.
\end{remark}

\section{Extremal values and the structure of extremizers for \texorpdfstring{$r_k(A)$}{rk(A)} when \texorpdfstring{$k$}{k} is odd}
\label{sec:extremal}

In this section, we consider the diagonal case in which all \(k\) sets are identical. For odd \(k\), we determine the exact maximum and minimum values of \(r_k(A)\) and characterize all sets attaining these extremal values. Theorem \ref{thm:higher-order-rearrangement} identifies an initial segment as a maximizer, while the following \emph{complement identity} allows us to determine the minimum.

\begin{lemma}
\label{lem:general-complement-expansion}
If \(A\subseteq \mathbb{Z}_2^n\), then
\begin{equation*}
    r_k(A)=\frac{|A|^k-(-1)^k |A^c|^k}{q}+
        (-1)^k r_k(A^c).    
\end{equation*}
\end{lemma}

\begin{proof}[Proof of Lemma \ref{lem:general-complement-expansion}]
For brevity, define $$f(a,b)\coloneqq r_k(\underbrace{A^c,\dots,A^c}_{a\text{ terms}},\underbrace{A,\dots,A}_{b\text{ terms}}),$$ where $a+b=k$.  By part 2 of Lemma \ref{lem: properties of r(A,B,C)}, for every $i\in [k]$, we have
$$f(i-1,k-i+1) + f(i,k-i)=|A^c|^{i-1}|A|^{k-i}.$$
Multiplying this identity by $(-1)^{i-1}$ and summing over $i\in [k]$, we obtain
\begin{align*}
    \sum_{i=1}^k (-1)^{i-1}f(i-1,k-i+1) + \sum_{i=1}^k (-1)^{i-1}f(i,k-i) = \sum_{i=1}^k (-1)^{i-1}|A^c|^{i-1}|A|^{k-i}.
\end{align*}
The expression on the left-hand side telescopes to $f(0,k) + (-1)^{k-1}f(k,0)$. Meanwhile, the expression on the right-hand side is a finite geometric sum equal to $\frac{|A|^k-(-1)^k|A^c|^k}{q}$. Since $f(0,k)=r_k(A)$ and $f(k,0)=r_k(A^c)$, the desired identity follows.
\end{proof}

We next compute \(r_k\) for an initial segment. Together with Theorem \ref{thm:higher-order-rearrangement}, this computation gives the maximum possible value of \(r_k(A)\) among sets of a prescribed size.

\begin{prop}
\label{cor:odd-initial-segment-value}
Let \(k\geq 3\) be odd, and let \(0< a\leq q\). Set $Q\coloneqq 2^{\lceil\log_2 a\rceil}$.
Then
\[
\max_{\substack{A\subseteq\mathbb{Z}_2^n\\ |A|=a}}r_k(A)
    =
    r_k(\mathrm{IS}_a)
    =
    \frac{a^k+(Q-a)^k}{Q}.
\]
\end{prop}

\begin{proof}[Proof of Proposition \ref{cor:odd-initial-segment-value}]
By Theorem~\ref{thm:higher-order-rearrangement}, $r_k(A)\leq r_k(\mathrm{IS}_a)$
for every \(A\subseteq\mathbb{Z}_2^n\) with \(|A|=a\). It therefore
remains to compute \(r_k(\mathrm{IS}_a)\).
If $a=1$, then $\mathrm{IS}_a=\{\mathbf{0}\}$, and the formula is immediate. Assume, therefore, that $a\geq 2$. Let $K\coloneqq \mathrm{IS}_Q$ and
$S\coloneqq K\setminus \mathrm{IS}_a$. Since $Q/2<a\leq Q$, the set
$S$ is contained in the upper half of $K$, which is a nontrivial
coset of an index-two subgroup of $K$. Since $k$ is odd, the sum of
$k$ elements of this coset cannot equal $\mathbf{0}$. Therefore, $r_k(S)=0$. Applying the complement identity to
$\mathrm{IS}_a$, regarded as a subset of the subgroup $K$, gives the desired formula.
\end{proof}

Having determined the maximum, we now turn to the minimum. Here, the assumption that \(k\) is odd is essential: the complement identity transforms the minimization problem for \(A\) into the maximization problem for \(A^c\).

%The case \(k=3\) of the following proposition was first established by Samotij and Sudakov \(\cite{SS16}\), using a result from spectral graph theory due to Alon and Chung \(\cite{alon}\). Our method is more direct and applies to every odd \(k\geq3\).

\begin{prop}
\label{prop:diagonal-minimum}
Let $k\geq 3$ be odd, and let $0\leq a<q$. Set $Q\coloneqq 2^{\lceil\log_2(q-a)\rceil}$. Then
\[
    \min_{\substack{A\subseteq\mathbb Z_2^n\\ |A|=a}} r_k(A)
    =
    r_k(\mathrm{FS}_a)
    =
    \frac{a^k+(q-a)^k}{q}
    -
    \frac{(q-a)^k+(Q-(q-a))^k}{Q}.
\]
\end{prop}

\begin{proof}[Proof of Proposition \ref{prop:diagonal-minimum}]
Since $k$ is odd, the complement identity gives
\[
    r_k(A)
    =
    \frac{a^k+(q-a)^k}{q}
    -
    r_k(A^c).
\]
Thus, minimizing $r_k(A)$ is equivalent to maximizing $r_k(A^c)$. By
Theorem~\ref{thm:higher-order-rearrangement},
$r_k(A^c)\leq r_k(\mathrm{IS}_{q-a}),$
with equality when $A=\mathrm{FS}_a$. Hence,
\[
    \min_{\substack{A\subseteq\mathbb Z_2^n\\ |A|=a}} r_k(A)
    =
    \frac{a^k+(q-a)^k}{q}
    -
    r_k(\mathrm{IS}_{q-a}).
\]
Applying Proposition~\ref{cor:odd-initial-segment-value} with $a$ replaced by $q-a$, we obtain
\[
    r_k(\mathrm{IS}_{q-a})
    =
    \frac{(q-a)^k+(Q-(q-a))^k}{Q},
\]
which proves the result.
\end{proof}

\begin{remark}\label{rem:evenk}
The diagonal problem for even \(k\) is substantially different from the odd case. Indeed, when \(k\) is even, the complement identity gives
\begin{equation*}
    r_k(A)=\frac{|A|^k-|A^c|^k}{q}+
        r_k(A^c).    
\end{equation*}
Thus, minimizing \(r_k(A)\) is equivalent to minimizing \(r_k(A^c)\), rather than maximizing it. The compression method in Theorem~\ref{thm:higher-order-rearrangement} determines the maximum possible value of \(r_k(A^c)\) and therefore does not resolve the diagonal minimization problem when \(k\) is even.

Already for $k=4$, the quantity
\begin{equation*}
        r_4(A)
        =
        |\{(a_1,a_2,a_3,a_4)\in A^4:
        a_1+a_2+a_3+a_4=\mathbf{0}\}|
\end{equation*}
coincides with the \emph{additive energy} of $A$. A simple counting argument gives
\begin{align*}
r_4(A)\ge |A|+\binom{4}{2}\binom{|A|}{2}=3|A|^2-2|A|,  
\end{align*}
where the right-hand side counts the trivial solutions. Equality holds if and only if \(A\) contains no four distinct elements
whose sum is zero. Equivalently, all sums $a+b$ with \(a,b\in A\)
and $a\neq b$ are distinct up to interchanging \(a\) and \(b\).
This is the \emph{Sidon set} condition in \(\mathbb Z_2^n\). 

Consequently, determining the range of $|A|$ for which this lower bound is sharp requires knowledge of the exact maximum size of a Sidon set in \(\mathbb Z_2^n\), which remains open in general; see \cite{CP26} and the references therein.
\end{remark}

Throughout the remainder of this section, we assume that $k\geq 3$ is odd. We now characterize completely the sets attaining the minimum and maximum possible values of \(r_k(A)\) among all \(a\)-element subsets \(A\subseteq\mathbb Z_2^n\).

Recall that a set $A\subseteq \mathbb{Z}_2^n$ is \emph{sum-free} if $r_3(A)=0$, or equivalently, if \((A+A)\cap A=\varnothing\). We begin with the following simple observation.

\begin{lemma}
\label{lem:rk-free-implies-sum-free}
If $S\subseteq\mathbb{Z}_2^n$ satisfies $r_k(S)=0$, then
$\mathbf{0}\notin S$ and $S$ is sum-free.
\end{lemma}

\begin{proof}[Proof of Lemma \ref{lem:rk-free-implies-sum-free}]
If $\mathbf{0}\in S$, then the $k$-tuple $(\mathbf{0},\dots,\mathbf{0})$ contributes to
$r_k(S)$, contradicting the assumption that $r_k(S)=0$. Suppose that $S$ is not sum-free. Then there
exist $\mathbf{x},\mathbf{y},\mathbf{x}+\mathbf{y}\in S$. Choosing any
$\mathbf{s}\in S$, the $k$-tuple
$\big(\mathbf{x},\ \mathbf{y},\ \mathbf{x}+\mathbf{y},\
\underbrace{\mathbf{s},\mathbf{s},\ldots, \mathbf{s},\mathbf{s}}_{k-3\text{ terms}}\big)
$ has sum $\mathbf{0}$, since $k-3$ is even. This again contradicts the assumption that $r_k(S)=0$.
\end{proof}

\begin{lemma}
\label{lem:partial-additivity-extension}
Let $K$ be a finite-dimensional vector space over $\mathbb{Z}_2$, let $S\subseteq K$ be a sum-free set, and let $D\coloneqq K\setminus S$.
Suppose that $f:D\to\mathbb{Z}_2$ satisfies
$f(x+y)=f(x)+f(y)$ whenever $ x,y,x+y\in D$. Then $f$ extends to a linear functional on $K$.
\end{lemma}
\begin{proof}[Proof of Lemma \ref{lem:partial-additivity-extension}]
Since $S$ is sum-free, we have $\mathbf{0}\notin S$, and hence
$\mathbf{0}\in D$. Taking $(x,y)=(\mathbf{0},\mathbf{0})$ gives
$f(\mathbf{0})=0$. Next, consider the linear span of $D$:
\[
    \langle D\rangle
    =
    \{d_1+\cdots+d_m:m\in\mathbb{N}_0,\ d_1,\ldots,d_m\in D\}.
\]
Define $\Phi:\langle D\rangle\to\mathbb{Z}_2$ as follows: if
$x=d_1+\cdots+d_m\in\langle D\rangle$, let $    \Phi(x)\coloneqq f(d_1)+\cdots+f(d_m).$
We claim that $\Phi$ is well-defined. It suffices to prove the following
claim.

\noindent
\textit{Claim.} If $d_1+\cdots+d_r=\mathbf{0}$ for some $r\geq 1$, then
$f(d_1)+\cdots+f(d_r)=0$.

We proceed by induction on $r$. The cases $r\leq 2$ are immediate, so
suppose that $r\geq 3$. We claim that there exists a pair $(i,j)$ with
$1\leq i<j\leq r$ such that $d_i+d_j\in D$. Indeed, otherwise $d_1+d_2, d_1+d_3,d_2+d_3$
would all belong to $S$. Since
$d_2+d_3=(d_1+d_2)+(d_1+d_3)$,
this would contradict the assumption that \(S\) is sum-free. Replacing
$d_i$ and $d_j$ by $d_i+d_j$ shortens the relation, so we may apply the
induction hypothesis. Since $d_i,d_j,d_i+d_j\in D$ and
$f(d_i+d_j)=f(d_i)+f(d_j)$, we conclude that
$f(d_1)+\cdots+f(d_r)=0$. This proves the claim.

It follows that $\Phi$ is a linear functional on $\langle D\rangle$.
Since $\langle D\rangle$ is a subspace of $K$, the function $\Phi$
extends to a linear functional on $K$. Thus,
$f$ extends to a linear functional on $K$.
\end{proof}

The following lemma shows that, for sets of a certain form, if a compression preserves \(r_k\), then the original set must have the same structure as its compression.

\begin{lemma}
\label{lem:inverse-compression-step}
Let $B\subseteq\mathbb{Z}_2^n$ have size $b>0$, and set $Q\coloneqq 2^{\lceil\log_2 b\rceil}.$
Let $\mathbf{v}\in\mathbb{Z}_2^n\setminus\{\mathbf{0}\}$, and suppose
that $\mathrm{C}_{\mathbf{v}}(B)=K\setminus S$, where
$K\leq\mathbb{Z}_2^n$, $|K|=Q$, and $S\subseteq K$ satisfies
$r_k(S)=0$. If $r_k(B)=r_k(\mathrm{C}_{\mathbf{v}}(B))$, then
$B=K'\setminus S'$ for some $K'\leq\mathbb{Z}_2^n$ with
$|K'|=Q$ and some $S'\subseteq K'$ satisfying $r_k(S')=0$.
\end{lemma}

\begin{proof}[Proof of Lemma \ref{lem:inverse-compression-step}]
We distinguish two cases.

1) Suppose that \(\mathbf{v}\in K\). Since each pair
\(\{\mathbf{h},\mathbf{h}+\mathbf{v}\}\) is either contained in \(K\) or
disjoint from \(K\), the inclusion \(\mathrm{C}_{\mathbf{v}}(B)\subseteq K\)
implies that \(B\subseteq K\). Let \(S'\coloneqq K\setminus B\).
Since \(|S'|=|S|=Q-b\), the complement identity
applied inside \(K\) gives
\[
    r_k(B)=\frac{b^k+(Q-b)^k}{Q}-r_k(S'),
    \qquad \text{and} \qquad
    r_k(\mathrm{C}_{\mathbf{v}}(B))
    =\frac{b^k+(Q-b)^k}{Q}-r_k(S).
\]
Since \(r_k(B)=r_k(\mathrm{C}_{\mathbf{v}}(B))\) and \(r_k(S)=0\), it follows that \(r_k(S')=0\). Thus, \(B=K\setminus S'\) has the required form.

2) Suppose that $\mathbf{v}\notin K$. We first make two observations:
\begin{itemize}
    \item For every $\mathbf{h}\in L_{\mathbf{v}}$, we have $       |K\cap\{\mathbf{h},\mathbf{h}+\mathbf{v}\}|\leq 1,$ since otherwise $\mathbf{v}\in K$.
    \item We have $\mathrm{C}_{\mathbf{v}}(B)\subseteq L_{\mathbf{v}}$. Indeed, if
    $\mathbf{h}+\mathbf{v}\in\mathrm{C} _{\mathbf{v}}(B)$, then $\mathbf{h}\in \mathrm{C}_{\mathbf{v}}(B)$, which would imply
    $\mathbf{v}\in K$.
\end{itemize}

It follows from the second observation and the definition of \(Q\) that
\[
    |K\cap L_{\mathbf{v}}|
    \geq |\mathrm{C}_{\mathbf{v}}(B)|
    =|B|
    >Q/2.
\]
Since $K\cap L_{\mathbf{v}}$ is a subspace of $K$, we must have
$K\cap L_{\mathbf{v}}=K$, and hence $K\subseteq L_{\mathbf{v}}$.

Let $D\coloneqq \mathrm{C}_{\mathbf{v}}(B)=K\setminus S$. For every
$\mathbf{h}\in D$, the set
$B\cap\{\mathbf{h},\mathbf{h}+\mathbf{v}\}$ is a singleton. Therefore,
there exists a function $f:D\to\mathbb{Z}_2$ such that $B\cap\{\mathbf{h},\mathbf{h}+\mathbf{v}\}
    =
    \{\mathbf{h}+f(\mathbf{h})\mathbf{v}\}$
for every $\mathbf{h}\in D$.

We claim that whenever
$\mathbf{x},\mathbf{y},\mathbf{x}+\mathbf{y}\in D$, we have
\begin{equation}
\label{eqn:f-additivity}
    f(\mathbf{x}+\mathbf{y})
    =
    f(\mathbf{x})+f(\mathbf{y}).
\end{equation}

Suppose, to the contrary, that \eqref{eqn:f-additivity} fails. Then there
exist $\mathbf{x}',\mathbf{y}',\mathbf{x}'+\mathbf{y}'\in D$ such that
$f(\mathbf{x}'+\mathbf{y}')\neq f(\mathbf{x}')+f(\mathbf{y}').$
By Lemma~\ref{lem:rk-free-implies-sum-free}, we have
$\mathbf{0}\notin S$, and hence $\mathbf{0}\in D$. By
Definition~\ref{def:fibre}, for every
$\mathbf{t}\in \{\mathbf{x}',\mathbf{y}',\mathbf{x}'+\mathbf{y}',\mathbf{0}\},$
the fibre $B_{\mathbf{t}}$ is the singleton $\{f(\mathbf{t})\}$.
Therefore,
\[
    r_k\Big(
        \mathrm{I}_{|B_{\mathbf{x}'}|},
        \mathrm{I}_{|B_{\mathbf{y}'}|},
        \mathrm{I}_{|B_{\mathbf{x}'+\mathbf{y}'}|},
        \underbrace{\mathrm{I}_{|B_{\mathbf{0}}|},\ldots,
        \mathrm{I}_{|B_{\mathbf{0}}|}}_{k-3\text{ terms}}
    \Big)
    =
    r_k(\mathrm{I}_1,\ldots,\mathrm{I}_1)
    =
    1.
\]
On the other hand, since $k-3$ is even,
\[
    r_k\Big(
        B_{\mathbf{x}'},
        B_{\mathbf{y}'},
        B_{\mathbf{x}'+\mathbf{y}'},
        \underbrace{B_{\mathbf{0}},\ldots,
        B_{\mathbf{0}}}_{k-3\text{ terms}}
    \Big)
    =
    0.
\]

By Lemma~\ref{lem: reduction from Z_2^n to Z_2} and
\eqref{h-fibre of C_v(X)},
\begin{align*}
r_k(B)
&=
\sum_{\substack{\mathbf{h}_1,\ldots,\mathbf{h}_k\in L_{\mathbf{v}}\\
\mathbf{h}_1+\cdots+\mathbf{h}_k=\mathbf{0}}}
r_k\big(B_{\mathbf{h}_1},\ldots,B_{\mathbf{h}_k}\big)\qquad \text{and} \qquad
r_k(\mathrm{C}_{\mathbf{v}}(B))
=
\sum_{\substack{\mathbf{h}_1,\ldots,\mathbf{h}_k\in L_{\mathbf{v}}\\
\mathbf{h}_1+\cdots+\mathbf{h}_k=\mathbf{0}}}
r_k\big(
    \mathrm{I}_{|B_{\mathbf{h}_1}|},\ldots,
    \mathrm{I}_{|B_{\mathbf{h}_k}|}
\big).
\end{align*}
By Lemma~\ref{lem: monotonicity of r_k for Z_2}, every summand in the
first sum is at most the corresponding summand in the second. However,
for the $k$-tuple $(\mathbf{x}',\mathbf{y}',\mathbf{x}'+\mathbf{y}',         \underbrace{\mathbf{0},\ldots,\mathbf{0}}_{k-3\text{ terms}}),$
the inequality is strict. Therefore,
$r_k(B)<r_k(\mathrm{C}_{\mathbf{v}}(B))$, contradicting the assumption that
$r_k(B)=r_k(\mathrm{C}_{\mathbf{v}}(B))$. This proves
\eqref{eqn:f-additivity}.

By Lemma~\ref{lem:rk-free-implies-sum-free}, the set $S$ is sum-free.
Hence, by Lemma~\ref{lem:partial-additivity-extension}, the function
$f$ extends to a linear functional $\ell$ on $K$. Define
$K'
    \coloneqq
    \{\mathbf{h}+\ell(\mathbf{h})\mathbf{v}:\mathbf{h}\in K\}$ and $S'
    \coloneqq
    \{\mathbf{h}+\ell(\mathbf{h})\mathbf{v}:\mathbf{h}\in S\}.$
The map $\phi:K\to\mathbb{Z}_2^n$ defined by $\phi(\mathbf{h})    \coloneqq\mathbf{h}+\ell(\mathbf{h})\mathbf{v}$ is linear. Moreover, it is injective: if $\phi(\mathbf{h})=\mathbf{0}$,
then $\mathbf{h}=\ell(\mathbf{h})\mathbf{v}$, and since
$K\subseteq L_{\mathbf{v}}$ while $\mathbf{v}\notin L_{\mathbf{v}}$,
we must have $\mathbf{h}=\mathbf{0}$. Therefore,
$K'\leq\mathbb{Z}_2^n$ and $|K'|=|K|=Q$. Moreover, $r_k(S')=0$ follows
from the injectivity of $\phi$ and the fact that $r_k(S)=0$.

Since $L_{\mathbf{v}}=(L_{\mathbf{v}}\setminus K)\sqcup(K\setminus S)\sqcup S,$ we have
\begin{align*}
B=
\Bigg(
\bigsqcup_{\mathbf{h}\in L_{\mathbf{v}}\setminus K}
B\cap\{\mathbf{h},\mathbf{h}+\mathbf{v}\}
\Bigg)
\sqcup
\Bigg(
\bigsqcup_{\mathbf{h}\in K\setminus S}
B\cap\{\mathbf{h},\mathbf{h}+\mathbf{v}\}
\Bigg)
\sqcup
\Bigg(
\bigsqcup_{\mathbf{h}\in S}
B\cap\{\mathbf{h},\mathbf{h}+\mathbf{v}\}
\Bigg).
\end{align*}
However,
$B\cap\{\mathbf{h},\mathbf{h}+\mathbf{v}\}=\varnothing$ whenever
$\mathbf{h}\in L_{\mathbf{v}}\setminus K$ or $\mathbf{h}\in S$.
Therefore,
\begin{align*}
B
=
\bigsqcup_{\mathbf{h}\in K\setminus S}
B\cap\{\mathbf{h},\mathbf{h}+\mathbf{v}\}=
\bigsqcup_{\mathbf{h}\in K\setminus S}
\{\mathbf{h}+f(\mathbf{h})\mathbf{v}\}
=
K'\setminus S',
\end{align*}
which completes the proof.
\end{proof}

We are now ready to prove Theorem \ref{thm:equality-maximum}, which characterizes the sets that maximize \(r_k\) among all subsets of \(\mathbb Z_2^n\) of a prescribed size.

\begin{proof}[Proof of Theorem \ref{thm:equality-maximum}]
\boxed{\Rightarrow} We first establish the following claim.

\noindent
\textit{Claim.}
The set $\mathrm{IS}_b$ can be written as $K\setminus S$, where
$K\leq\mathbb{Z}_2^n$, $|K|=Q$, and $S\subseteq K$ satisfies
$|S|=Q-b$ and $r_k(S)=0$.

Indeed, let $K\coloneqq \mathrm{IS}_Q$ and  $S\coloneqq\mathrm{IS}_Q\setminus\mathrm{IS}_b$. Since $Q/2<b \leq Q$, we have $|S|=Q-b$, and $S$ is contained in the upper half of $K$. As in the proof of Proposition~\ref{cor:odd-initial-segment-value}, the oddness of $k$ implies that $r_k(S)=0$, proving the claim.

Now suppose that $B$ is a maximizer. We distinguish two cases.

1) Suppose that $\mathrm{C}_{\mathbf{v}}(B)=B$ for every
$\mathbf{v}\in\mathbb{Z}_2^n\setminus\{\mathbf{0}\}$. Then, by
Lemma~\ref{lem:compression-sets-initial}, we have
$B=\mathrm{IS}_b$, and the conclusion follows from the claim.

2) Suppose that there exists
$\mathbf{v}\in\mathbb{Z}_2^n\setminus\{\mathbf{0}\}$ such that
$\mathrm{C}_{\mathbf{v}}(B)\neq B$. By the compression argument used in the proof of
Theorem~\ref{thm:majorization}, there exists a finite
sequence of compressions
\[
    B=B_0\longrightarrow B_1\longrightarrow\cdots
    \longrightarrow B_t=\mathrm{IS}_b
\]
that transforms $B$ into $\mathrm{IS}_b$, where $   r_k(B_{i-1})\leq r_k(B_i)$ for every $i\in [t]$. Since $B$ is a maximizer, we have
$r_k(B_{i-1})=r_k(B_i)$ for every $i\in[t]$. By the claim, the terminal
set $\mathrm{IS}_b$ has the desired form. Applying
Lemma~\ref{lem:inverse-compression-step} repeatedly along the compression
sequence in reverse order shows that the original set $B$ also has the
required form.

\(\boxed{\Leftarrow}\) Conversely, suppose that \(B=K\setminus S\), where
\(K\leq\mathbb{Z}_2^n\), \(|K|=Q\), \(|S|=Q-b\), and \(r_k(S)=0\).
Since \(k\) is odd, the complement identity applied inside \(K\) gives
\[
    r_k(B)=\frac{b^k+(Q-b)^k}{Q}.
\]
By the claim, the same argument shows that $r_k(B) = r_k(\mathrm{IS}_b)$. Since \(\mathrm{IS}_b\) maximizes \(r_k\) by Theorem~\ref{thm:higher-order-rearrangement}, \(B\) also maximizes $r_k$.
\end{proof}

As an immediate corollary, we obtain a characterization of the sets that minimize \(r_k\) among all subsets of \(\mathbb Z_2^n\) of a prescribed cardinality. 
\begin{cor}
\label{cor:equality-minimum}
Let $A\subseteq\mathbb{Z}_2^n$ have size $a<q$. Set
$Q\coloneqq 2^{\lceil\log_2(q-a)\rceil}.$
Then $A$ minimizes $r_k(X)$ among all $a$-element subsets
$X\subseteq\mathbb{Z}_2^n$ if and only if $A=K^c\sqcup S,$
where $K\leq\mathbb{Z}_2^n$, $|K|=Q$, and $S\subseteq K$ satisfies $|S|=Q-(q-a)$ and $r_k(S)=0.$
\end{cor}
\begin{proof}[Proof of Corollary \ref{cor:equality-minimum}]
Assume that $a<q$. Since $k$ is odd, the complement identity shows that $A$ minimizes $r_k(X)$ if and only if $A^c$ maximizes $r_k(Y)$ among all $(q-a)$-element subsets $Y\subseteq \mathbb{Z}_2^n$. The result now follows from Theorem~\ref{thm:equality-maximum},
since $A^c=K\setminus S$ if and only if $A=K^c\sqcup S$.
\end{proof}

\end{document}